\documentclass[12pt,a4paper]{amsart}
\calclayout

\let\le\undefined
\DeclareMathSymbol{\le}{\mathrel}{AMSa}{"36}         %\leqslant
\let\ge\undefined
\DeclareMathSymbol{\ge}{\mathrel}{AMSa}{"3E}         %\geqslant

\DeclareMathOperator{\chr}{char}
\DeclareMathOperator{\Gal}{Gal}
\DeclareMathOperator{\Ann}{Ann}
\DeclareMathOperator{\Hom}{Hom}
\DeclareMathOperator{\Tor}{Tor}
\DeclareMathOperator{\Br}{Br}
\DeclareMathOperator{\tr}{tr}

\newcommand{\Z}{{\mathbb Z}}
\newcommand{\Q}{{\mathbb Q}}
\newcommand{\F}{{\mathbb F}}
\newcommand{\R}{{\mathbb R}}
\newcommand{\C}{{\mathbb C}}

\def\ot{\DOTSB\otimes}
\def\sub{\DOTSB\subset}
\def\rarrow{\DOTSB\longrightarrow}

\def\lrarrow{\DOTSB\.\relbar\joinrel\relbar\joinrel\rightarrow\.}
\newcommand{\bu}{\bullet}
\newcommand{\dsb}{\dotsb}
\newcommand{\dsc}{\dotsc}
\newcommand{\dsm}{\dotsm}

\newcommand{\oK}{\,\overline{\!K}}
\newcommand{\KM}{\mathrm{K}^{\mathrm{M\mskip-.2\thinmuskip\relax}}}

\newcommand{\+}{\nobreakdash-}
\renewcommand{\:}{\colon}
\renewcommand{\;}{,\medspace}
\renewcommand{\.}{\mskip .5\thinmuskip\relax}

\newcommand{\gr}{{\mathrm{gr}}}
\newcommand{\q}{{\mathrm{q}}}
\newcommand{\ab}{{\mathrm{ab}}}

\newcommand{\Section}[1]{\bigskip\section{#1}\medskip}
\theoremstyle{plain}
\newtheorem*{kc}{Koszulity Conjecture}
\newtheorem{mkc}{Module Koszulity Conjecture}
\newtheorem*{thm}{Theorem}
\newtheorem*{thm1}{Theorem 1}
\newtheorem*{thm2}{Theorem 2}
\newtheorem*{lem}{Lemma}
\newtheorem*{lem1}{Lemma 1}
\newtheorem*{lem2}{Lemma 2}
\newtheorem*{lem3}{Lemma 3}
\newtheorem*{prop}{Proposition}
\theoremstyle{definition}
\newtheorem*{rem}{Remark}

\begin{document}

\title{Galois Cohomology of a Number Field is Koszul}
\author{Leonid Positselski}

\address{Faculty of Mathematics and Laboratory of Algebraic
Geometry, National Research University Higher School of Economics,
Moscow 117312; and \newline
\indent Sector of Algebra and Number Theory, Institute for
Information Transmission Problems, Moscow 127994, Russia}

\email{posic@mccme.ru}

\begin{abstract}
 We prove that the Milnor ring of any (one-dimensional) local or
global field $K$ modulo a prime number~$l$ is a Koszul algebra
over~$\Z/l$.
 Under mild assumptions that are only needed in the case $l=2$,
we also prove various module Koszulity properties of this algebra.
 This provides evidence in support of Koszulity conjectures for
arbitrary fields that were proposed in our previous papers.
 The proofs are based on the Class Field Theory and computations
with quadratic commutative Gr\"obner bases (commutative PBW-bases).
\end{abstract}

%\keywords{Global fields, local fields, Galois cohomology,
%Koszul algebras, Koszul modules, Class Field Theory, Chebotarev's
%density theorem, filtrations on algebras, commutative PBW-bases,
%commutative Gr\"obner bases}

\maketitle

\tableofcontents

\section*{Introduction}
\medskip

\setcounter{subsection}{-1}
\subsection{{}}
 Let $K$ be a field and $l\ne\chr K$ be a prime number.
 The well-known \emph{Milnor--Bloch--Kato} conjecture claims
that the natural morphism of graded $\Z/l$\+algebras, called
the \emph{Galois symbol}, or the \emph{norm residue homomorphism},
$$
 \KM(K)/l\lrarrow \textstyle\bigoplus_n H^n(G_K,\mu_l^{\ot n})
$$
is an isomorphism.
 Here $\KM(K)$ denotes the Milnor K\+theory ring of $K$, \ 
$G_K=\Gal(\oK/K)$ is the absolute Galois group of $K$, and
$\mu_l$ is the group of $l$\+roots of unity in $\oK$.
 For algebraic number fields and their functional analogues,
this conjecture was proven by J.~Tate in~\cite{Tate}; see
also~\cite{BT}.
 For arbitrary fields, it was established by A.~Merkurjev and
A.~Suslin~\cite{MS} in the degree $n=2$; the prolonged work on
a complete proof was recently finished by M.~Rost, V.~Voevodsky,
and collaborators~\cite{Voev}.

\subsection{{}}
 Another approach to proving this conjecture was suggested in
our paper~\cite{PV}.
 There it was shown that the Milnor--Bloch--Kato conjecture would
follow from its low-degree part if one knew the quadratic algebra
$\KM(K)/l$ to be \emph{Koszul}.
 The argument in~\cite{PV} was only directly applicable to
fields $K$ having no algebraic extensions of degree prime to~$l$;
it is well-known that it suffices to prove the Milnor--Bloch--Kato
conjecture for such fields.
 The proper scope of the Koszulity conjecture was demonstrated
in~\cite{Pober,Partin}, where a motivic interpretation of it was
found in the case when $K$ contains a primive $l$\+root of unity.

 A positively graded associative algebra $A=k\oplus A_1\oplus
A_2\oplus\dsb$ over a field~$k$ is called \emph{Koszul}
\cite[Subsection~2.2]{PV} if the groups $H_{ij}(A)=\Tor_{ij}^A(k,k)$
vanish for all $i\ne j$, where the first grading $i$ is
the homological grading and the second grading $j$ is
the \emph{internal} grading, which is induced from the grading of~$A$.
 This definition, introduced originally for algebras with
(locally) finite-dimensional components by S.~Priddy~\cite{Pr}
(see~\cite{BGSoe} or~\cite{PP} for a detailed treatment), was
extended to the infinite-dimensional case in~\cite{PV}.
 In particular, the conditions on $H_{ij}(A)$ for $i=1$ and~$2$
mean that the algebra $A$ is quadratic, i.~e., generated by $A_1$
with relations of degree~$2$.

\begin{kc}
 For any field $K$ containing a primitive root of unity of a prime
degree~$l$, the algebra $\KM(K)/l$ is Koszul.
\end{kc}

 Another formulation of the Koszulity conjecture is the assertion
that the Galois cohomology algebra $\bigoplus_n H^n(G_K,\mu_l^{\ot n})$
is Koszul, in the same assumptions on $K$ and~$l$.
 In view of the Merkurjev--Suslin theorem, this formulation clearly
implies the Milnor--Bloch--Kato conjecture.
 By the result of~\cite{PV}, the former formulation also implies
the Milnor--Bloch--Kato, provided that one knows just a little bit
more than the Merkurjev--Suslin theorem.

\subsection{{}}
 Further and stronger versions of the Koszulity conjecture were
proposed in the papers~\cite{Pbogom,Pdivis}.
 Assume that either $l$~is odd and the field $K$ contains a primitive
$l$\+root of unity, or $l=2$ and the field $K$ contains a square
root of~$-1$ (we always presume that $\chr K\ne 2$ when speaking of
the square root of~$-1$).
 Then the Milnor algebra $\KM(K)/l$ is the quotient algebra of
the exterior algebra $\Lambda_{\Z/l}(K^*/K^{*l})$ by its ideal $J_K$
generated by the Steinberg symbols.
 It was shown in~\cite{Pbogom} that if the ideal $J_K$ is
a \emph{Koszul module} (in the appropriately shifted grading) over
the algebra $\Lambda_{\Z/l}(K^*/K^{*l})$, then both Koszulity of
the algebra $\KM(K)/l$ and a certain version of Bogomolov's freeness
conjecture for the field $K$ follow.

 A positively graded left module $M=M_1\oplus M_2\oplus\dsb$ over
a Koszul algebra~$A$ is called \emph{Koszul}
\cite[Subsection~3.3]{Pbogom} if the groups $H_{ij}(A,M)=
\Tor_{ij}^A(k,M)$ vanish for all $i\ne j-1$ (see~\ref{pbw-theorem}
for comments on the grading conventions).
 This definition, which first appeared in~\cite{BGSoe}, was studied
in detail in~\cite{PP,Pbogom} (see also~\cite{BF}).
 In particular, the conditions on $H_{i,j}(A,M)$ for $i=0$ and~$1$
mean that the $A$\+module $M$ is quadratic, i.~e., generated by $M_1$
with relations in degree~$2$.
 The hypothesis of Koszulity of the ideal $J_K$ can be equivalently
restated as follows.

\begin{mkc}
 For any field $K$ and a prime number~$l$ such that $K$ contains
a primitive $l$\+root of unity if\/ $l$~is odd and $K$ contains
a square root of~$-1$ if\/ $l=2$,
the $\Lambda_{\Z/l}(K^*/K^{*l})$\+module $\KM_+(K)/l =
\KM_1(K)/l\oplus\KM_2(K)/l\oplus\dsb$ is Koszul.
\end{mkc}
 
\subsection{{}}
 Assume that $K$ contains a primitive $l$\+root of unity, and let
$c\in K^*$ be an element not belonging to $K^{*l}$.
 Let $L=K[\sqrt[l]{c}]$.
 Assume further that the Milnor--Bloch--Kato conjecture holds for
the fields $K$ and $L$ and the algebra $\KM(K)/l$ is Koszul.
 Then it was shown in~\cite[proof of Corollary~15]{Pdivis} together
with \cite[Theorem~6.1]{Pbogom} that the algebra $\KM(L)/l$ is Koszul
provided that the annihilator ideal $\Ann(c\bmod l)\sub\KM(K)/l$
of the element $(c\bmod l)\in\KM_1(K)/l$ is a Koszul module over
$\KM(K)/l$.

 It was conjectured that the annihilator ideal in $\KM(K)/l$ of
any nonzero element in $\KM_1(K)/l$ is a Koszul module.
 The silly filtration conjecture for Artin--Tate motives with
$\Z/l$\+coefficients related to the field extension $L/K$ follows
from this module Koszulity conjecture~\cite[Subsection~9.8]{Partin}.
 It can be equivalently restated as follows.

\begin{mkc}
 For any field $K$ containing a root of unity of a prime degree~$l$
and any element $c\in\KM_1(K)/l$, the ideal $(c)=c\.\KM(K)/l$ is
a Koszul module over the algebra $\KM(K)/l$.
\end{mkc}

 Another equivalent formulation of the same conjecture is that
the quotient algebra $(\KM(K)/l)/(c\.\KM(K)/l)$ is a Koszul
module (in the appropriately shifted grading) over
$\KM(K)/l$.

\subsection{{}}
 In this paper, we prove all these Koszulity conjectures for all
local and global fields, i.~e., the algebraic extensions of $\R$, \
$\Q_p$, \ $\F_p((z))$, \ $\Q$, or $\F_p(z)$.
 The only exception is that our proof of Module Koszulity Conjecture~2
in the case of a global field depends on the assumption that
$\{c,c\}=0$ in $\KM_2(K)/l$.
 This assumption always holds when $l$ is odd or $K$ contains
a square root of~$-1$.
 Proving the assertions of these conjectures in the case when $K$
does not contain a primitive $l$\+root of unity is relatively
easy, and we do so, even though there are few general reasons to
believe in the Koszulity conjectures in such a case.
 For local fields, we also prove a certain extension of Module
Koszulity Conjecture~1 to the case of fields not necessarily
containing a square root of~$-1$ when $l=2$.

 This paper is a far-reaching extension of the appendix to~\cite{PV}.
 Our methods involve the construction of infinite quadratic
commutative Gr\"obner bases (infinite commutative PBW-bases) in
the algebras $\KM(K)/l$.
 These constructions make heavy use of the descriptions of Galois
cohomology provided by the local and global class field theory,
some further results from the global class field theory, and
the Chebotarev density theorem.
 These methods also allow to obtain a new proof of the fact that
the graded algebra $\bigoplus_n H^n(G_K,\mu_l^{\ot n})$ is quadratic
for a local or global field~$K$.

\subsection{{}}
 The general formalism of filtrations and PBW-bases indexed by
well-ordered sets is developed in Section~1.
 The relevant background facts from the algebraic number theory
are recalled and discussed in Section~2.
 Koszul properties of the Milnor algebra/Galois cohomology of
a local field are established in Section~3.
 Koszulity of the ideal of Steinberg relations $J_K$ for a global
field $K$ containing a primitive $l$\+root of unity if $l$ is odd
and a square root of $-1$ if $l=2$ is proven in Section~4.
 Koszulity of the algebra $\KM(K)/l$ for any global field $K$
containing a primitive $l$\+root of unity is demonstrated in
Section~5.
 Koszulity of the annihilator ideals in Milnor algebras of global
fields is shown, under certain assumptions, in Section~6.
 All the mentioned Koszul properties are proven for a global field
$K$ not containing a primitive $l$\+root of unity in Section~7.

\subsection*{Acknowledgment}
 The author is grateful to Vladimir Voevodsky for posing the problem.
 This project would never have a chance to succeed without
the invaluable participation of Alexander Vishik in its early stages.
 The work was largely done when I was a graduate student at
Harvard University in the Fall of 1995, and I want to thank Harvard
for its hospitality.
 The author was partially supported by a grant from P.~Deligne
2004 Balzan prize and RFBR grants while finalizing the arguments
and preparing the manuscript.
 I~was also visiting Weizmann Institute of Science during the later
part of this time in~2014.
 Finally, I want to thank the referee for reading the manuscript
carefully and making a number of helpful suggestions.

\Section{Preliminaries on PBW-Bases}

\subsection{Well-ordered sets and filtrations}  \label{well-ordered}
 Let $k$ be a field, $V$ be a $k$\+vector space, and
$I=\{\alpha\}$ be a well-ordered set.
 An \emph{increasing filtration $F$ on $V$ with values in $I$} is
a family of subspaces $F_\alpha V\sub V$, \ $\alpha\in I$, such that
$F_\alpha V\sub F_\beta V$ for $\alpha<\beta$ and
$V=\bigcup_{\alpha\in I}F_\alpha V$.
 The \emph{associated quotient} space $\gr^FV=\bigoplus_\alpha
\gr^F_\alpha V$ is the $I$\+graded vector space with the components
$\gr^F_\alpha V = F_\alpha V\big/\bigcup_{\beta<\alpha}F_\beta V$.

\begin{lem}
 Let $C^\bu$ be a complex of $I$\+filtered vector spaces (with
the differentials preserving the filtrations).
 Then one has $H^0(C^\bu)=0$ provided that $H^0(\gr^F C^\bu)=0$.
 Conversely, if $H^{-1}(\gr^F C^\bu)=0=H^1(\gr^F C^\bu)$ and
$H^0(C^\bu)=0$, then $H^0(\gr^F C^\bu)=0$.
\end{lem}

\begin{proof}
 Straightforward induction on the well-ordering.
\end{proof}

\subsection{Graded ordered semigroups}  \label{tensor-filtered}
 A \emph{graded ordered semigroup}~\cite[Section~7 of Chapter~4]{PP}
\,$\Gamma$ is a collection of well-ordered sets $\Gamma_n$, \
$n\in\Z_{\ge0}$, endowed with associative multiplication maps
$\Gamma_n\times\Gamma_m\rarrow\Gamma_{n+m}$ strictly compatible
with the orderings, i.~e., $\alpha<\beta$ implies
$\alpha\gamma<\beta\gamma$ and $\gamma\alpha<\gamma\beta$ for all
$\alpha$, $\beta\in\Gamma_n$ and $\gamma\in\Gamma_m$.
 In addition, we assume that the only element of $\Gamma_0$ is the unit
of the semigroup~$\Gamma$.

 Let $U$ be a $\Gamma_n$\+filtered vector space and $V$ be
a $\Gamma_m$\+filtered vector space over~$k$; we will denote both
filtrations by~$F$.
 Define a $\Gamma_{n+m}$\+valued filtration on the tensor product
$U\ot_k V$ by the rule
$$
 F_\gamma(U\ot_kV) = \textstyle\sum_{\alpha\beta\le\gamma}
 F_\alpha U\ot F_\beta V,
$$
where $\alpha\in\Gamma_n$, \ $\beta\in\Gamma_m$, and
$\gamma\in \Gamma_{n+m}$.
 Similarly, if $U=\bigoplus_{\alpha\in\Gamma_n}U_\alpha$ and
$V=\bigoplus_{\beta\in\Gamma_m} V_\beta$ are a $\Gamma_n$\+
and $\Gamma_m$\+graded vector spaces, then the tensor product
$U\ot_k V$ is a $\Gamma_{n+m}$\+graded vector space with
the components
$$
 (U\ot_k V)_\gamma = \textstyle\bigoplus_{\alpha\beta=\gamma}
 U_\alpha\ot_k V_\beta.
$$

\begin{lem}
 There is a natural isomorphism of\/ $\Gamma_{n+m}$\+graded vector
spaces
$$
 \gr^{\Gamma_{n+m}}(U\ot_kV)\simeq
 \gr^{\Gamma_n}U\ot_k\gr^{\Gamma_m}V.
$$
\end{lem}

\begin{proof}
 To define a natural map $\gr^{\Gamma_n}U\ot_k\gr^{\Gamma_m}V
\rarrow\gr^{\Gamma_{n+m}}(U\ot_kV)$, choose for any classes
$\bar u\in\gr^F_\alpha U$ and $\bar v\in\gr^F_\beta V$ their
representatives $u\in F_\alpha U$ and $v\in F_\beta V$,
and assign the class $\overline{u\ot v}\in \gr^F_{\alpha\beta}
(U\ot_k V)$ of the element $u\ot v\in F_{\alpha\beta}
(U\ot_k V)$ to the tensor product $\bar u\ot\bar v\in
\gr^F_\alpha U\ot_k\gr^F_\beta V$.
 Checking that this is a well-defined isomorphism is easy.
\end{proof}

\subsection{Filtered algebras and modules}
\label{filtered-algebras-modules}
 Let $A=\bigoplus_{n=0}^\infty A_n$ be a graded associative
$k$\+algebra with $A_0=k$ and $\Gamma$ be a graded ordered semigroup.
 A \emph{$\Gamma$\+valued filtration} on $A$ is a family of
filtrations $F$ with values in $\Gamma_n$ on the grading
components $A_n$ of the algebra $A$ such that the multiplication
maps $A_n\ot_k A_m\rarrow A_{n+m}$ are compatible with
the $\Gamma_{n+m}$\+valued filtrations.
 By Lemma~\ref{tensor-filtered}, the associated graded quotient
vector space $\gr^FA=\bigoplus_n\gr^FA_n$ has a natural
structure of a graded $k$\+algebra.

 Let $M=\bigoplus_{n=1}^\infty M_n$ be a graded left module
over~$A$.
 A $\Gamma$\+valued filtration on $M$ compatible with the given
$\Gamma$\+valued filtration on $A$ is a family of filtrations
$F$ with values in $\Gamma_n$ on the grading components $M_n$
of the module $M$ such that the multiplication maps $A_n\ot_k M_m
\rarrow M_{n+m}$ are compatible with the $\Gamma_{n+m}$\+valued
filtrations.
 The associated graded quotient vector space $\gr^FM=\bigoplus_n
\gr^FM_n$ has a natural structure of a graded left module
over $\gr^FA$.

 Recall the definitions of the homology of positively graded
associative algebras and modules
from~\cite[Subsections~2.2--3]{Pbogom}:
$H_{i,j}(A)=\Tor_{i,j}^A(k,k)$ and $H_{i,j}(A,M)=\Tor_{i,j}^A(k,M)$.
 Here the second grading~$j$ comes from the grading of $A$ and $M$;
the grading~$i$ is the \emph{homological} grading and
the grading~$j$ is called the \emph{internal} grading.

 When $A$ and $M$ are endowed with $\Gamma$\+valued filtrations,
the bar-complexes computing $H_*(A)$ and $H_*(A,M)$ acquire
the induced filtrations.
 The components of the bar-complexes of the internal grading~$n$
are filtered with values in~$\Gamma_n$.
 By Lemma~\ref{tensor-filtered}, the associated quotient complexes
of these bar-complexes with respect to these filtrations are
the bar-complexes computing $H_*(\gr^FA)$ and $H_*(\gr^FA\;\gr^FM)$.
 By Lemma~\ref{well-ordered}, it follows that $H_{i,j}(A)=0$
provided that $H_{i,j}(\gr^FA)=0$, and $H_{i,j}(A,M)=0$ provided
that $H_{i,j}(\gr^FA\;\gr^FM)=0$.

\subsection{PBW-Theorem}  \label{pbw-theorem}
 A positively graded associative algebra $A=\bigoplus_{n=0}^\infty A_n$
is called \emph{quadratic} if it is isomorphic to the quotient algebra
of a tensor (free associative) algebra $\bigoplus_{n=0}^\infty V^{\ot n}$
by an ideal generated a vector subspace $R\sub V\ot_k V$.
 The \emph{quadratic part} of a positively graded algebra $A$
is the quadratic algebra $\q A$ with the space of generators $V=A_1$
and the subspace of quadratic relations $R=\ker(A_1\ot_k A_1\to A_2)$.
 The natural morphism of graded algebras $\q A\rarrow A$ is
an isomorphism in degree~$1$ and a monomorphism in degree~$2$;
it is an isomorphism of graded algebras if and only if the graded
algebra $A$ is quadratic.

 Similarly, a positively graded left module $M=\bigoplus_{n=1}^\infty
M_n$ over a quadratic algebra $A=\bigoplus_{n=0}^\infty V^{\ot n}/(R)$
is called \emph{quadratic} if it is isomorphic to the quotient module
of a free left $A$\+module $A\ot_k U$ generated in degree~$1$ by
a submodule generated by a vector subspace $P\sub V\ot_k U$.
 The \emph{quadratic part} of a positively graded module $M$ over
a positively graded algebra~$A$ is the quadratic module $\q_AM$
with the space of generators $U=M_1$ and the subspace of relations
$P=\ker(A_1\ot M_1\to M_2)$ over the quadratic algebra~$\q A$.
 The natural morphism of graded $\q A$\+modules $\q_AM\rarrow M$ is
an isomorphism in degree~$1$ and a monomorphism in degree~$2$;
a positively graded module $M$ over a quadratic algebra $A$ is
quadratic if and only if this morphism is an isomorphism in all
the degrees~\cite[Subsection~3.1]{Pbogom}.

 Here we use a convention for graded modules slightly different
from that in~\cite{PP,Pbogom}: quadratic modules $M$ are generated
by $M_1$ with relations in degree~$2$.
 The definition of a Koszul module from~\cite[Subsection~3.3]{Pbogom}
is modified accordingly: a positively graded module $M$
over a Koszul algebra $A$ is called Koszul if $H_{i,j}(A,M)=0$
for all $i\ne j-1$.
 As it was explained in~\ref{filtered-algebras-modules}, given
compatible $\Gamma$\+valued filtrations $F$ on $A$ and $M$,
the algebra $A$ is quadratic or Koszul whenever the algebra $\gr^FA$
is, and the $A$\+module $M$ is quadratic or Koszul whenever
the $\gr^FA$\+module $\gr^FM$ is.
 The following theorem is a more delicate result in this direction.
 It is a generalization of the quadratic case of the Diamond Lemma
for Gr\"obner bases~\cite{Ber,Buch}.

\begin{thm}
 Let $A$ be a $\Gamma$\+filtered graded algebra and $M$ be
a $\Gamma$\+filtered graded $A$\+module.
 Then
\begin{enumerate}
\renewcommand{\theenumi}{\arabic{enumi}}
\item if the algebra $A$ is quadratic, the algebra $\gr^FA$
is generated in degree~$1$, the algebra $\q\,\gr^FA$ is Koszul,
and the natural map $\q\,\gr^F A\rarrow \gr^FA$ is an isomorphism
in the degree $n=3$, then the algebras $\gr^FA$ and $A$ are Koszul;
\item if the algebra $\gr^FA$ is Koszul, the $A$\+module $M$ is
quadratic, the $\gr^FA$\+module $\gr^FM$ is generated in
degree~$1$, the $\gr^FA$\+module $\q_{\gr^F\!\.A}\,\gr^FM$ is Koszul,
and the natural map $\q_{\gr^F\!\.A}\,\gr^FM\rarrow\gr^FM$ is
an isomorphism in the degree $n=3$, then the $\gr^FA$\+module
$\gr^FM$ and the $A$\+module $M$ are Koszul.
\end{enumerate}
\end{thm}

\begin{proof}
 One only has to prove that the algebra $\gr^FA$ or the module
$\gr^FM$ is quadratic. 
 Proceed by induction on $n\ge4$ showing the map
$\q\,\gr^F A\rarrow \gr^FA$ or $\q_{\gr^F\!\.A}\,\gr^FM\rarrow\gr^FM$
is an isomorphism in degree~$n$.
 For this purpose, consider the component of internal
degree~$n$ of the initial fragment of the bar-complex
$A_+^{\ot 4}\rarrow A_+^{\ot 3}\rarrow A_+^{\ot 2}\rarrow A_+\rarrow k$
or $A_+^{\ot 3}\ot_k M\rarrow A_+^{\ot 2}\ot_k M\rarrow A_+\ot_k M
\rarrow M\rarrow0$ (given that $n\ge4$, this fragment is
acyclic when the algebra $A$ or the module $M$ are Koszul).
 Then apply the second assertion of Lemma~\ref{well-ordered}
in order to conclude that the associated quotient complex with respect
to the $\Gamma_n$\+valued filtration, which is isomorphic to
the similar complex for the algebra $\gr^FA$ or
the $\gr^FA$\+module $\gr^FM$, has zero cohomology at
the middle term (cf.~\cite[Theorem~7.1 from Chapter~4]{PP}).
 Indeed, the condition $H_{3,n}(\q\,\gr^FA)=0$ or
$H_{2,n}(\q\,\gr^FA,\,\q_{\gr^F\!\.A}\,\gr^FM)=0$ together with
the induction assumption guarantee that the associated graded
complex has no cohomology at the second term.
 The assumption of generation in degree~$1$ tells that it has
no cohomology at the fourth term, and the quadraticity assumption
means that the filtered complex has no cohomology at the middle term.
\end{proof}

\subsection{Inverse lexicographical ordering}  \label{inverse-lexi}
 In this paper we will use graded ordered semigroups $\Gamma$ of
the following special form.
 As a graded semigroup, $\Gamma$ is isomorphic to the free
commutative semigroup generated by the set $\Gamma_1$.
 So elements of $\Gamma_n$ are the commutative monomials
$\alpha_1^{n_1}\dsm\alpha_m^{n_m}$, where $\alpha_1<\dsb<\alpha_m$
are elements of $\Gamma_1$ and $n_1+\dsb+n_m=n$.
 The order on $\Gamma_n$ is the \emph{inverse lexicographical
order}: $\alpha_1^{n_1'}\dsm\alpha_m^{n_m'}<\alpha_1^{n_1''}\dsm
\alpha_m^{n_m''}$ if there exists $1\le j\le m$ such that
$n'_i=n''_i$ for $i<j$ and $n'_j>n''_j$.
 For example, $\alpha_1\alpha_4<\alpha_2\alpha_3$ if $\alpha_1<
\alpha_2<\alpha_3<\alpha_4$.

 We will be interested in $\Gamma$\+valued filtrations $F$ on
graded algebras $A$ such that the associated quotient spaces
$\gr^F_\alpha A_1$ are one-dimensional for all $\alpha\in\Gamma_1$.
 Abusing terminology, we will speak of $\Gamma_1$\+indexed bases
$\{x_\alpha\}$ of $A_1$, presuming that $x_\alpha\in F_\alpha A_1$
has a nonzero image, which will be also denoted by $x_\alpha$, in
$\gr^F_\alpha A_1$.
 So the basis $\{x_\alpha\}$ in $A_1$ will be only defined up to
an upper-triangular linear transformation.
 Similarly, we will consider $\Gamma$\+valued filtrations $F$ on
graded modules $M$ such that the associated quotient spaces
$\gr^F_\alpha M_1$ are no more than one-dimensional for all
$\alpha\in\Gamma_1$.

\subsection{Koszulity of algebras}  \label{algebras-koszulity}
 All our graded algebras $A$ will be associative and unital,
generated by $A_1$, and either commutative (when $\chr k=2$)
or supercommutative with respect to the parity associated with
the grading (when $\chr k$ is odd).
 Here a graded algebra $A$ is called supercommutative if the identity
$a^2=0$ for all $a\in A_n$ with odd~$n$ holds in $A$, together with
the identity $ab=(-1)^{nm}ba$ for $a\in A_n$ and $b\in A_m$.
 Notice that any supercommutative algebra over a field~$k$ of
characteristic~$2$ is commutative, but not the other way.

 We will also assume the graded algebra $\gr_FA$ to be
generated in degree~$1$.
 Once a $\Gamma_1$\+indexed filtration of $A_1$ is fixed, this
condition defines a unique extension of this filtration to
a $\Gamma$\+valued filtration of~$A$.
 Given a graded algebra $A$ as above with a $\Gamma$\+valued
filtration $F$ satisfying this condition together with the conditions
of~\ref{inverse-lexi}, the associated quotient algebra $\gr^FA$ is
a commutative or supercommutative \emph{monomial} algebra.
 In other words, it is the quotient algebra of the free commutative
or supercommutative algebra generated by the elements $x_\alpha$
of degree~$1$ by an ideal generated by a set of monomials
in~$x_\alpha$.
 This is easy to see; actually, no other relations can be compatible
with the $\Gamma$\+grading of $\gr^FA$.

 The quadratic part $\q\,\gr^FA$ of a (super)commutative monomial
algebra $\gr^FA$ is always Koszul.
 Indeed, when the set of generators $\{x_\alpha\}$ is finite,
this is the result of the paper~\cite{Fr} (see
also~\cite[Theorem~8.1 of Chapter~4]{PP}),
and the general case follows by passing to the inductive limit
of subalgebras generated by finite subsets of $\{x_\alpha\}$.
 By Theorem~\ref{pbw-theorem}(1), if $A$ is quadratic and $\gr^FA$
has no relations of degree~$3$, then $\gr^FA$ is quadratic.
 By the result of~\ref{filtered-algebras-modules}, if $\gr^FA$ is
quadratic, then $A$ is Koszul.

 When $\gr^FA$ is quadratic, the basis in $A$ formed by those
monomials in $x_{\alpha}$ that survive in $\gr^FA$ is called
a \emph{commutative PBW\+basis} of~$A$.

\begin{rem}
 The commutative PBW-bases of (super)commutative algebras, which
are used in this paper, are particular cases of commutative
Gr\"obner bases~\cite{Buch} whose application to Koszulity questions
is based on the result of R.~Fr\"oberg's paper~\cite{Fr}.
 These are different from noncommutative PBW-bases, which are
particular cases of noncommutative Gr\"obner bases~\cite{Ber} and
whose application to Koszulity was worked out already
by Priddy in~\cite{Pr} (see also~\cite[Sections~1\+-5 of
Chapter~4]{PP}).
 In application to commutative algebras, the commutative PBW-bases
are generally more powerful.
\end{rem}

\subsection{Koszulity of ideals of relations}  \label{relations-koszul}
 Let $B$ be a Koszul algebra and $J\sub B$ be a two-sided ideal
concentrated in the degrees $n\ge 2$.
 Set $A=B/J$.
 Then $J$ is a Koszul left $B$\+module (in the grading appropriately
shifted by~$1$) if and only if $A_+=A_1\oplus A_2\oplus A_3\oplus\dsb$
is a Koszul left $B$\+module.
 In this case, by~\cite[Corollary~6.2(c)]{Pbogom}, the algebra $A$
itself is Koszul.

 Now let $\Lambda$ be the exterior (free supercommutative) algebra
over a field~$k$ generated by a set of elements $x_\alpha$
of degree~$1$ and $A=\Lambda/J$ be the quotient algebra of $\Lambda$
by an ideal of monomials of degree $n\ge 2$ in~$x_\alpha$.
 Let $T$ denote the set of all quadratic monomials $x_\alpha x_\beta$
that are nonzero in $A$; consider $T$ as the set of edges of
an (infinite) graph with the set of vertices~$\{x_\alpha\}$.

\begin{prop}
 Assume that $A_n=0$ for $n\ge3$.  Then 
\begin{enumerate}
\renewcommand{\theenumi}{\arabic{enumi}}
\item the monomial algebra $A$ is Koszul if and only if the graph
$T$ contains no triangles;
\item the $\Lambda$\+module $A_+$ is Koszul if and only if
the graph $T$ contains no cycles of any (finite) length.
\end{enumerate}
\end{prop}

\begin{proof}
 Part~(1): it is clear that $A$ is quadratic if and only if $T$
does not contain triangles.
 Since $A$ is supercommutative monomial, it is Koszul whenever it
is quadratic (see~\ref{algebras-koszulity} and~\cite{Fr}
or~\cite[Theorem~8.1 of Chapter~4]{PP}).

 Part~(2), ``if'': it suffices to consider the case when the set
$\{x_\alpha\}$ is finite, since then one can pass to the inductive
limit of the similar modules over finitely generated subalgebras
of~$\Lambda$.
 In the finitely generated case, proceed by induction in the number
of vertices.
 Choose a vertex $x_\alpha$ with a single edge $x_\alpha x_\beta$
coming out of it.
 The $k$\+vector subspace spanned by $x_\alpha$ and $x_\alpha
x_\beta$ is a $\Lambda$\+submodule in $A_+$ which is easily seen
to be Koszul.
 Set $(x_\alpha)=\Lambda x_\alpha$.
 The quotient module $M=A_+/\langle x_\alpha,x_\alpha x_\beta \rangle$
is a Koszul module over $\Lambda/(x_\alpha)$ by the induction
assumption.
 Since the ideal $(x_\alpha)\sub\Lambda$ is a Koszul $\Lambda$\+module
(cf.~\ref{annihilator-koszul}), it follows by the way of
the spectral sequence
$E^2_{p,q}=\Tor_p^{\Lambda/(x_\alpha)}(H_q(\Lambda,\Lambda/(x_\alpha)),M)
\Longrightarrow H_{p+q}(\Lambda,M)$ \cite[(6.1)]{Pbogom} that
$A_+/\langle x_\alpha,x_\alpha x_\beta \rangle$ is a Koszul module
over $\Lambda$, too.

 Part~(2), ``only if'': let $\Gamma$ be the free commutative
semigroup generated by the set of indices $\{\alpha\}$; then
$\Lambda$ and $A_+$ are $\Gamma$\+graded.
 Considering subcomplexes of the bar-complex consisting of all
the $\Gamma$\+grading components corresponding to a subsemigroup
of $\Gamma$ spanned by a subset of $\{\alpha\}$, one can see
that the $\Lambda$\+module $A_+$ corresponding to a graph $T$
is Koszul if and only if the same is true for any full subgraph
of $T$ (i.~e., any subgraph consisting of all $T$\+edges between
a given subset of vertices).
 So it suffices to consider the case when $T$ is a finite polygon
with $n$ vertices and $n$ edges.
 In this case, the homology exact sequence related to the short
exact sequence of $\Lambda$\+modules $A_2\rarrow A_+\rarrow A_1$
shows that $\dim H_{n-2,n}(\Lambda,A_+)=1$.
\end{proof}

\subsection{Koszulity of annihilator ideals} \label{annihilator-koszul}
 Let $A$ be a commutative or supercommutative Koszul algebra and
$c\in A_1$ be a nonzero element.
 Then the annihilator ideal $\Ann(c)=\{\.a\in A\mid ac=0\.\}\sub A$
is a Koszul $A$\+module if and only if the ideal $(c)=Ac\sub A$ is
a Koszul $A$\+module, and if and only if the quotient algebra
$A/(c)$ is a Koszul $A$\+module in the grading shifted by~$-1$.

 Assume that there exists a commutative PBW-basis in $A$ corresponding
to a well-ordered set of generators $\{x_\alpha\}\in A_1$ such that
the minimal element of this set is $x_0=c$.
 Then the ideal $(c)$ is a Koszul $A$\+module.
 This is true due to our particular choice of the inverse
lexicographical ordering of monomials in~$x_\alpha$.

 Indeed, let the $\Gamma$\+valued filtration $F$ on $(c)\sub A$ be
induced from the $\Gamma$\+valued filtration $F$ on~$A$.
 Then the ideal $\gr^F(c)\sub\gr^F A$ is generated by the class
$\bar c\in\gr^F A_1$ of the element~$c$ (because $c$ is the minimal
element in the set of generators $\{x_\alpha\}$ and the order of
monomials is inverse lexicographical).

 This is an ideal in a (super)commutative quadratic monomial algebra
generated by a subset of the algebra generators.
 All such ideals are Koszul.
 Indeed, for a finitely generated monomial algebra, this is shown
in~\cite[proof of Theorem~8.1 of Chapter~4]{PP}, and the general
case follows by passing to the inductive limit of ideals in
finitely generated monomial algebras.

\Section{Preliminaries on Number Fields}

 First of all let us recall that for any field $K$ and a prime
number~$l$ the multiplication in $\KM(K)/l$ is supercommutative when
$l$~is odd or $K$ contains a square root of~$-1$, and commutative
when $l=2$.
 More precisely, one has $\{x,x\}=\{-1,x\}$ in $\KM_2(K)$ for any
$x\in \KM_1(K)$ \cite[Section~I.1]{BT}.

\setcounter{subsection}{-1}
\subsection{Equal characteristics}  \label{equal-char}
 For any field $K$ of prime characteristic~$p$ such that $[K:K^p]\le p$
one has $\KM_n(K)/p=0$ for $n\ge 2$ \cite[Proposition~I.5.13]{BT}.
 This includes any finite extensions of $\F_p((z))$ or $\F_p(z)$.
 For any field $K$ of characteristic~$p$, one has
$H^n(G_K,\Z/p)=0$ for $n\ge 2$ \cite[n$^\circ$\,II.2.2]{Ser}.

\subsection{Finite and archimedean fields}  \label{archimedean}
 For a finite field $K=\F_q$, one has $\KM_n(K)=0=
H^n(G_K,\mu_l^{\ot n})$ for any $n\ge2$ and any prime $l$
not dividing~$q$ \cite[Corollary~I.5.12]{BT}.

 For the field of complex numbers $K=\C$, one has $\KM_n(K)/l=0$
for any $n\ge 1$ and any~$l$.
 For the field of real numbers $K=\R$, one has $\KM_n(K)/l=0$ for any
$n\ge 1$ and any odd~$l$, while $\KM(K)/2\simeq\Z/2[\{-1\}]\simeq
H(G_K,\Z/2)$ is the polynomial ring with one generator of degree~$1$
corresponding to the class of the element $-1\in K^*$.

\subsection{Discrete valuation fields}  \label{discrete-valuation}
 Let $K$ be a Henselian discrete valuation field and $k$ be its
residue field.
 Let $l\ne\chr k$ be a prime number.
 Then the $\Z/l$\+algebra $\KM(K)/l$ is generated by
the $\Z/l$\+algebra $\KM(k)/l$ and an element $\{\pi\}\in\KM_1(K)/l$,
corresponding to any uniformizing element $\pi\in K$, subject
to the relations of supercommutativity of $\{\pi\}$ with
$\KM(k)/l$ and $\{\pi,\pi\}=\{-1,\pi\}$ \cite[Proposition~I.4.3]{BT}.

 The absolute Galois group $G_K$ is an extension of the semidirect
product of $G_k$ with the group $\Z_l$, where $G_k$ acts by
the cyclotomic character, by a group of order prime to~$l$.
 This allows to obtain a similar description of the algebra
$\bigoplus_n H^n(G_K,\mu_l^{\ot n})$ in terms of the algebra
$\bigoplus_n H^n(G_k,\mu_l^{\ot n})$.

\subsection{Nonarchimedean local fields}  \label{local-fields}
 Let $K$ be a finite extension of $\Q_p$ or $\F_p((z))$ and
$l\ne\chr K$ be a prime number.
 The computation of $\KM_n(K)/l\simeq H^n(G_K,\mu_l^{\ot n})$
\cite[Corollary on page~268]{Tate} is provided by the local
class field theory.

 When $K$ does not contain a primitive $l$\+root of unity,
one has $\KM_n(K)/l=0$ for $n\ge2$.
 When $K$ contains a primitive $l$\+root of unity, one has
$\KM_n(K)/l=0$ for $n\ge 3$, and $\KM_2(K)/l\simeq\mu_l$.
 In the latter case, the multiplication map
$$
 \KM_1(K)/l\ot_{\Z/l}\KM_1(K)/l\lrarrow \KM_2(K)/l
$$
is a nondegenerate pairing.
 This pairing provides the comparison between the isomorphism
$G_K^\ab/l\simeq K^*/K^{*l}$ of the local class field theory
and the isomorphism $G_K^\ab/l\simeq \Hom_\Z(K^*,\mu_l)$ of
the Kummer theory; hence the nondegeneracy.

 The $\Z/l$\+vector space $\KM_1(K)/l$ is finite-dimensional.
 Except when $K$ is a finite extension of $\Q_l$, its dimension
is equal to~$2$ when $K$ contains a primitive $l$\+root of unity,
and~$1$ otherwise.
 When $K$ is a finite extension of $\Q_l$, the dimension is~$\ge3$
when $K$ contains a primitive $l$\+root of unity, and $\ge2$ otherwise.

 When $l$ is odd and $K$ contains a primitive $l$\+root of unity,
or $l=2$ and $K$ contains a square root of~$-1$, \ $\KM_1(K)/l$ is
a symplectic vector space with respect to the multiplication pairing.
 In other words, the multiplication pairing is skew-symmetric, i.~e.,
$\{x,x\}=0$ for all $x\in\KM_1(K)/l$.
 In particular, the dimension of $\KM_1(K)/l$ is even.

 When $l=2$ and $K$ does not contain a square root of~$-1$, it
follows from nondegeneracy and the relation $\{x,x\}=\{-1,x\}$
that $\dim\KM_1(K)/l$ is even when $\{-1,-1\}=0$ and odd when
$\{-1,-1\}\ne0$.
 In both cases, the isomorphism class of the (nonskew-symmetric)
pairing form is determined by $\dim\KM_1(K)/l$.

 Except when $K$ is a finite extension of~$\Q_l$, the product $\{x,y\}$
of the classes of two elements $x$, $y\in K^*$ with the logarithmic
valuations $v(x)=0=v(y)$ is always zero in $\KM_2(K)/l$, and 
the product of the classes of two elements $x\in K^*\setminus K^{*l}$
and $y\in K^*$ with $v(x)=0$ and $v(y)$ not divisible by~$l$ is 
always nonzero in $\KM_2(K)/l$.

\subsection{Global fields with root of unity}  \label{with-root}
 For any field $K$ and a prime number $l\ne\chr K$, the group
$H^2(G_K,\mu_l)$ is isomorphic to the subgroup ${}_l\!\.\Br K$ of
the Brauer group $\Br K$ consisting of all elements annihilated by~$l$.

 For any finite extension $K$ of $\Q$ or $\F_q(z)$ and a prime number
$l\ne\chr K$ the natural map $\KM_2(K)/l\rarrow H^2(K,\mu_l^{\ot 2})$
is an isomorphism \cite[Theorem~5.1]{Tate}.
 Assuming that $K$ contains a primitive $l$\+root of unity and
combining this isomorphism with the computation of $\Br K$ provided
by the global class field theory, we see that there is a natural
short exact sequence
$$
 0\rarrow \KM_2(K)/l\lrarrow\textstyle\bigoplus_v\KM_2(K_v)/l\.\simeq\.
 \bigoplus_{v'}\mu_l\lrarrow\mu_l\rarrow0.
$$
 Here the direct sum in the second term is over all valuations~$v$
of $K$, and the direct sum in the third term is taken over all
the valuations~$v'$ not including the complex valuations if $l=2$, or
not including the archimedean valuations if $l$ is odd.
 The rightmost map is the simple summation over~$v'$.
 The assertion that the composition of the two maps vanishes is one
of the formulations of the reciprocity law.

 Given two elements $x$ and $y$ in $K^*$, or $K_v^*$, or
$K_v^*/K_v^{*l}$, etc., we will denote by $\{x,y\}_v$ their product
in $\KM_2(K_v)/l$, and identify the latter group with $\mu_l$
(assuming that we are not in the case when $\KM_2(K_v)/l=0$).
 So the reciprocity law takes the form $\sum_v\{x,y\}_v=0$ for
any $x$, $y\in K^*$.

 For any $n\ge3$ and any global field $K$, the natural map
$$
 \KM_n(K)\lrarrow\textstyle\bigoplus_v \KM_n(K_v)/2
$$
is an isomorphism~\cite[Theorem~II.2.1(3)]{BT}.
 Here the summation is over all the real valuations~$v$ of~$K$.
 One can obtain a compatible description of $H^n(G_K,\mu_l^{\ot n})$
from the global class field theory, by computing $H^n(G_K,\mu_l)$
in terms of $H^n(G_K,K^*)$ and the latter in terms of
$H^n(G_{K_v},K_v^*)$ and the cohomology of the classes of id\`eles.

\subsection{Exceptional set of valuations}  \label{exceptional-set}
 Let $K$ be a finite extension of $\Q$ or $\F_q(z)$ containing
a primitive $l$\+root of unity.
 Let $S$ be a finite set of valuations of $K$ containing all
the archimedean valuations and all the valuations lying over~$l$,
and generating the class group of the field~$K$.
 Denote by $W_S$ the $\Z/l$\+vector space
$$
 W_S=\textstyle\bigoplus_{v\in S} K_v^*/K_v^{*l},
$$
and let $K_S\sub K^*$ denote the subgroup of all elements $b$ having
the logarithmic valuation $p(b)=0$ (in other words, $b$ is integral
and integrally invertible in $K_p$) for all valuations $p\notin S$.

\begin{lem1}
 The natural map $K_S/l\rarrow W_S$ is injective.
\end{lem1}

\begin{proof}
 This is \cite[Lemma~VII.9.2 and Remark VII.9.3]{CF}.
\end{proof}

 Define a $\mu_l$\+valued bilinear form on $W_S$ as the orthogonal
sum of the bilinear forms on $K_v^*/K_v^{*l}$, that is
$(x,y)_S=\sum_{v\in S}\{x_v,y_v\}_v$.

\begin{lem2}
 The subspace $K_S/l\sub W_S$ coincides with its own orthogonal
complement with respect to the pairing form $({-},{-})_S$.
\end{lem2}

\begin{proof}
 The pairing form on $W_S$ is symmetric or skew-symmetric and
nondegenerate, since the pairings on $K_v^*/K_v^{*l}$ are.
 By the reciprocity law, one has $(K_S/l,K_S/l)=0$.
 It remains to check that $\dim  W_S = 2\dim  K_S/l$.
 We will show that $\dim K_S/l=\#S$ and $\dim W_S=
2\#S$, where $\#S$ is the number of elements in~$S$. 

 Indeed, by Dirichlet's unit theorem the group $K_S$ is the direct
sum of a free abelian group of rank $\#S-1$ and the finite cyclic
group of roots of unity in $K$, whose order is divisible by~$l$
by assumption.
 This computes $\dim K_S/l$.

 To compute $\dim W_S$, consider two cases separately.
 When $K$ is a finite extension of $\F_q(z)$, one has
$\dim K_v^*/K_v^{*l}=2$ for all $v\in S$.
 When $K$ is a finite extension of $\Q$, one has
\begin{itemize}
\item $\dim K_v^*/K_v^{*l}=2$ for all nonarchimedean $v\in S$
not lying over~$l$;
\item $\dim K_v^*/K_v^{*l}=2+[K_v:\Q_l]$ for any nonarchimedean
$v\in S$ lying over~$l$;
\item $\dim K_v^*/K_v^{*l}=0$ when $K_v=\C$; and
\item $\dim K_v^*/K_v^{*l}=1$ when $K_v=\R$, since such valuations~$v$
can only exist when $l=2$, as $\R$ does not contain any other
$l$\+roots of unity. 
\end{itemize}
 Summing this up, one easily obtains $\dim W_S=2\#S$.
\end{proof}

 The following lemma can be thought of as a kind of ``approximation
theorem for id\`eles modulo~$l$''.

\begin{lem3}
 Let $w$ be an element of $W_S$ and $D$ be a divisor of $K$ supported
outside of $S$, i.~e., a formal linear combination of valuations of
$K$, not belonging to $S$, with integral coefficients.
 The pairing with $w$ defines a $\mu_l$\+valued linear function on
$K_S/l$, and another such function is provided by the linear
combination of Frobenius elements in $\Gal(K[\sqrt[l]{K_S}]/K)$
corresponding to the divisor~$D$.
 Suppose that these two linear functions coincide.
 Then there exists an element $a\in K^*$ whose image in $W_S$ is equal
to~$w$ and whose divisor outside $S$ is equal to~$D$.
 Furthermore, the element~$a$ is unique modulo $K_S^l$.
\end{lem3}

\begin{proof}
 Since $S$ generates the class group of $K$, one can find an element
$b\in K^*$ whose divisor outside of $S$ is equal to~$D$.
 Let us denote the image of $b$ in $W_S$ also by~$b$.
 By the reciprocity law, the element $w/b\in W_S$ is orthogonal to
$K_S/l$, hence by Lemma~2 it belongs to $K_S/l$.
 Lift it to an element $c\in K_S$ and set $a=bc$.
 The uniqueness follows immediately from Lemma~1.
\end{proof}

\subsection{Symplectic case}  \label{symplectic-case}
 The following lemma is useful in the case of a global field $K$
which contains a primitive $l$\+root of unity when $l$~is odd,
or contains a square root of~$-1$ when $l=2$.
 Recall that the pairing $({-},{-})_S$ is skew-symmetric in
this case.

\begin{lem}
 Suppose that a (finite-dimensional) symplectic vector space $W$ over
a field~$k$ is decomposed into an orthogonal direct sum of symplectic
vector spaces $W_v$.
 Let $L$ be a Lagrangian subspace in~$W$.
 Then there exist Lagrangian subspaces $M_v$ in $W_v$ such that
the direct sum of $M_v$ is complementary to $L$ in~$W$.
\end{lem}

\begin{proof}
 It suffices to consider the case when all $W_v$ are two-dimensional
(otherwise decompose every one of them into an orthogonal direct sum
of two-dimensional symplectic vector spaces).
 Order the set of indices $\{v\}$ and proceed by induction, choosing
subspaces $M_{v'}\sub W_{v'}$ such that $\bigoplus_{v'\le v}M_{v'}$
does not intersect~$L$.
 Assume that $M_{v'}$ have been chosen so that $\bigoplus_{v'<v}M_{v'}$
does not intersect~$L$.
 Then, since $L$ is Lagrangian, there exists at most one line
$N_v\sub W_v$ for which $N_v\oplus\bigoplus_{v'<v}M_{v'}$
intersects~$L$.
 So a line $M_v\sub W_v$ such that $\bigoplus_{v'\le v}M_{v'}$
does not intersect~$L$ can always be chosen.
\end{proof}

\subsection{Global fields without root of unity}  \label{without-root}
 Let $K$ be a finite extension of $\Q$ or $\F_q(z)$ that does not
contain a primitive $l$\+root of unity.
 Notice that $l$~is necessarily odd in this case, so $\KM(K)/l$
is supercommutative and $\KM_n(K)/l=0$ for $n\ge3$.

 Set $L=K[\sqrt[l]{1}]$.
 The degree $[L:K]$ of this field extension divides~$l-1$ and
is prime to~$l$.
 Passing to $\Gal(L/K)$\+invariants in the description of
$\KM_2(L)/l\simeq H^2(G_L,\mu_l^{\ot 2})$ given in~\ref{with-root},
one concludes that the $\Z/l$\+vector space $\KM_2(K)/l\simeq
H^2(G_K,\mu_l^{\ot 2})$ is isomorphic to the direct sum of
the groups $\mu_{l,v}$ of $l$\+roots of unity in $K_v$ over all
the nonarchimedean valuations $v$ of $K$ for which $K_v$
contains a primitive $l$\+root of unity.

 So the product $\{x,y\}$ in $\KM_2(K)/l$ of the classes of two
elements $x$, $y\in K^*$ can be considered as the collection of
the local products $\{x,y\}_v\in\KM_2(K_v)/l\simeq\mu_{l,v}$
indexed by all such valuations~$v$.
 There are no relations between the local products: any finite
collection of elements in $\mu_{l,v}$ corresponds to an element
of $\KM_2(K)/l$.

 Let $S'$ and $S$ be finite sets of valuations of $K$ and $L$,
respectively, such that $S$ is the set of all valuations of $L$
lying over the valuations of $K$ beloning to~$S'$, the set $S'$
contains all the archimedean valuations of $K$ and all
the valuations, lying over~$l$, the set $S'$ generates the class
group of $K$, and the set $S$ generates the class group of~$L$.
 Set $W_S=\bigoplus_{v\in S} L_v^*/L_v^{*l}$, and let
$W'_{S'}$ denote the direct sum of $K_v^*/K_v^{*l}$ over all
the valuations $v\in S'$ such that $K_v$ contains a primitive
$l$\+root of unity.
 There is a natural injective map $W'_{S'}\rarrow W_S$.
 Set $M=L[\sqrt[l]{L_S}]$.

\begin{lem}
 For any element $w'\in W'_{S'}$ there exist infinitely many
valuations~$p$ of $K$ outside $S'$ for which $K_p$ does not
contain a primitive $l$\+root of unity and there exists
an element $a_p\in K^*$ whose image in $W'_{S'}$ is equal
to~$w'$ and whose divisor outside $S'$ is equal to~$p$.
\end{lem}

\begin{proof}
 Let $D'$ be a divisor of the field $K$ outside of $S'$; one can
naturally assign to it a divisor $D$ of the field $L$ outside of~$S$.
 For any element $w'\in W'_{S'}$, consider its image $w\in W_S$,
the pairing with~$w$ in $W_S$ as a linear function $L_S/l\rarrow
\mu_l$, and this linear function as an element of $\Gal(M/L)$.
 If the linear combination of Frobenius elements in $\Gal(M/L)$
corresponding to $D$ is equal to this element, then
by Lemma~\ref{exceptional-set}.3 there exists a unique, up to
$L_S^l$, element $a_2\in L^*$ whose image in $W_S$ is equal to~$w$
and whose divisor outside of $S$ is equal to~$D$.
 Due to the uniqueness, this element defines a $\Gal(L/K)$\+invariant
class in $L^*/L^{*l}$.

 From the short exact sequence $\mu_l\rarrow L^*\rarrow L^{*l}$,
Hilbert's Theorem~90, and the order of the group $\Gal(L/K)$ being
prime to~$l$, one can see that $H^1(\Gal(L/K)\;L^{*l})=0$.
 Hence there exists a $\Gal(L/K)$\+invariant element $a_1\in L^*$
which differs from $a_2$ by an element of $L^{*l}$.
 The element $a_1$ belongs to $K^*$, and its image in $W'_{S'}$
is equal to~$w'$, since its image in $W_S$ is equal to~$w$.
 The divisor of $a_1\in K^*$ is congruent to $D'$ modulo~$l$,
since the divisor of $a_1\in L^*$ is congruent to $D$ modulo~$l$ and
the ramification indices in the extension $L/K$ are prime to~$l$.
 Since $S'$ generates the class group of $K$, one can multiply $a_1$
with an element of $K^{*l}$ so that the resulting element $a\in K$
has the divisor $D'$ outside of $S'$; clearly, the image of~$a$
in $W'_{S'}$ is equal to~$w'$.

 The linear combination of Frobenius elements in $\Gal(M/L)$ 
corresponding to $D$ is the image of the linear combination of
Frobenius elements in $\Gal(M/K)^\ab$ corresponding to $D'$
under the transfer map $\tr\:\Gal(M/K)^\ab\rarrow\Gal(M/L)$.
 Since the element $h\in\Gal(M/L)$ corresponding to~$w$ is invariant
under $\Gal(L/K)$ and the order of the latter group is prime to~$l$,
the element $h$ is equal to the transfer of the element
$g\in\Gal(M/K)^\ab$ obtained as the image of $h/[L:K]$ under
the map $\Gal(M/L)\rarrow\Gal(M/K)^\ab$.

 Being an extension of the abelian groups $\Gal(L/K)$ and $\Gal(M/L)$
of coprime orders, the group $\Gal(M/K)$ is their semidirect
product.
 So the group $\Gal(L/K)$ can be embedded into $\Gal(M/K)$.
 Choose a nontrivial element $f\in\Gal(L/K)$ and consider
the product $q=(h/[L:K])f\in\Gal(M/K)$.
 Its image in $\Gal(M/K)^\ab$ is the product of~$g$ with
the image of~$f$, which we will denote also by~$f$.
 For the reasons of orders of the elements, or commutation of
transfer with the corestriction in group homology, it is clear
that $\tr(f)=1$ and $\tr(gf)=\tr(g)=h$ in $\Gal(L/K)$.

 By Chebotarev's density theorem, there exist infinitely many
valuations~$p$ of the field $K$ outside of $S'$ whose Frobenius
elements in $\Gal(M/K)$ are conjugate to~$q$.
 In this case, the Frobenius element of~$p$ in $\Gal(L/K)$ is
nontrivial, so $K_p$ does not contain a primitive $l$\+root of unity.
 Furthermore, let $D$ be the divisor of $L$ outside $S$ equal to
the image of~$p$ (which is considered as a divisor of~$K$).
 Then the linear combination of Frobenius elements in $\Gal(M/L)$
corresponding to $D$ is equal to the pairing with~$w$ in $W_S$
as a linear function $L_S/l\rarrow\mu_l$.
 Hence the element $a_p=a\in K^*$ constructed above has the desired
properties.
\end{proof}

\Section{Koszulity for Local Fields}  \label{local-koszulity}

 For any field $K$ and a prime number $l$ denote by $\Lambda(K,l)$
the graded algebra over $\Z/l$ generated by $\Lambda_1(K,l)=K^*/K^{*l}$
with the relations $\{x,-x\}=0$ for $x\in K^*$.
 The algebra $\Lambda(K,l)$ is always Koszul.

 Indeed, when $l$ is odd or $K$ contains a square root of $-1$,
this algebra is simply the exterior algebra generated by
$K^*/K^{*l}$.
 Otherwise, choose any well-ordered basis $\{x_\alpha\}$ of
the $\Z/l$\+vector space $K^*/K^{*l}$ such that the first basis
vector is $x_0=\{-1\}$, and consider the related $\Gamma$\+valued
filtration $F$ of $\Lambda(K,l)$
(see~\ref{inverse-lexi}--\ref{algebras-koszulity}).
 Then the algebra $\gr^F\Lambda(K,l)$ is isomorphic to the tensor
product of the symmetric algebra with one generator $\{-1\}$ and
the exterior algebra generated by $\Lambda_1(K,l)/
\langle\{-1\}\rangle$.

 There is a natural morphism of graded $\Z/l$\+algebras $\Lambda(K,l)
\rarrow\KM(K)/l$.
 Let $J_K$ denote its kernel; it is the ideal generated by
the Steinberg symbols.

\begin{thm1} \
\begin{enumerate}
\renewcommand{\theenumi}{\arabic{enumi}}
\item
 Let $K$ be an algebraic extension of\/ $\R$, \ $\Q_p$, or\/
$\F_p((z))$, and\/ $l$~be a prime number.
 Then the ideal $J_K\sub\Lambda(K,l)$ is a Koszul
$\Lambda(K,l)$\+module (in the grading shifted by~$1$).
 In particular, the algebra $\KM(K)/l$ is Koszul.
\item
 Let $K$ be a Henselian discrete valuation field with the residue
field\/~$k$ and\/ $l\ne\chr k$ be a prime number.
 Then the algebra $\KM(K)/l$ is Koszul whenever the algebra
$\KM(k)/l$ is.
 The ideal $J_K$ is a Koszul $\Lambda(K,l)$\+module whenever
the ideal $J_k$ is a Koszul $\Lambda(k,l)$\+module.
\end{enumerate}
\end{thm1}

\begin{proof}
 Part~(1): the cases $K\supset\R$ or $l=\chr K$ are trivial in
view of~\ref{archimedean} and~\ref{equal-char}, respectively.
 Indeed, in the former case one has $J_K=0$, and in the latter case
one can use the fact that the left $A$\+module $A_{\ge2}=
A_2\oplus A_3\oplus\dsb$ is Koszul for any Koszul algebra~$A$.
 Passing to the inductive limit, one reduces the problem to
the case when $K$ is a finite extension of $\Q_p$ or $\F_p((z))$
(and $l\ne\chr K$).

 The case when $K$ does not contain a primitive $l$\+root of unity
is similar to the above; see~\ref{local-fields}.
 When $l$ is odd and $K$ contains a primitive $l$\+root of unity,
or $l=2$ and $K$ contains a square roof of~$-1$, one can choose any
ordered basis of $K^*/K^{*l}$; the corresponding supercommutative
monomial algebra $\gr^F\.\KM(K)/l$ obviously satisfies the condition
of Proposition~\ref{relations-koszul}(2), since the graph $T$
contains only one edge.

 When $l=2$, the class of~$-1$ is nontrivial in $\KM_1(K)/2$, but
$\{-1,-1\}=0$ in $\KM_2(K)/2$, choose any ordered basis $\{x_\alpha\}$
of $K^*/K^{*2}$ with the minimal element $x_0=\{-1\}$ and the second
minimal element $x_1$ such that $x_0x_1\ne0$.
 Consider the related $\Gamma$\+valued filtrations $F$ of
$\Lambda(K,2)$ and $\KM(K)/2$.
 The corresponding algebra $\gr^F\Lambda(K,2)$ is the tensor product
of the symmetric algebra generated by $x_0$ and the exterior algebra
in other variables.
 The algebra $\gr^F\.\KM(K)/2$ is the commutative monomial algebra
with the only nonzero monomial $x_0x_1$ in the degrees~$\ge2$.
 One easily checks that both the $\gr^F\Lambda(K,2)$\+module spanned
by $x_0$ and $x_0x_1$ and the quotient module of $\gr^F\.\KM_+(K)/2$
by this submodule are Koszul.

 When $l=2$ and $\{-1,-1\}\ne0$ in $\KM_2(K)/2$, choose any ordered
basis $\{x_\alpha\}$ of $K^*/K^{*2}$ in which the minimal three
elements $x_0$, $x_1$, $x_2$ are such that $x_1=\{-1\}$ and
$x_0^2=x_0x_1=0\ne x_0x_2$ in $\KM_2(K)/2$.
 Consider the related $\Gamma$\+valued filtrations $F$ of
$\Lambda(K,2)$ and $\KM(K)/2$.
 The corresponding algebra $\gr^F\.\KM(K)/2$ is the commutative
monomial algebra with the only nonzero monomial $x_0x_2$ in
the degrees~$\ge2$.
 The algebra $\gr^F\Lambda(K,2)$ is the quadratic commutative
monomial algebra with the defining relations $x_0x_1=0$ and
$x_\alpha^2=0$ for $\alpha\ge2$.
 As above, one easily checks that the graded
$\gr^F\Lambda(K,2)$-module $\gr^F\.\KM_+(K)/2$ is Koszul.

 Part~(2): the argument is based on the description of $\KM(K)/l$
given in~\ref{discrete-valuation}.
 This time, the increasing filtrations on graded algebras that
we need to use are indexed by the conventional integers.
 Set $F_0\KM(K)/l=\KM(k)/l$ and $F_1\KM(K)/l=\KM(K)/l$.
 Then the graded algebra $\gr^F\.\KM(K)/l$ is the supertensor product
of $\KM(k)/l$ and the exterior algebra with one generator $\{\pi\}$
in degree~$1$, hence it is Koszul provided that $\KM(k)/l$
is \cite[Corollary~1.2 of Chapter~3]{PP}.

 To prove the second assertion, define also a compatible increasing
filtration $F$ on $\Lambda(K,l)$ by the rule $F_0\Lambda(K,l)=
\Lambda(k,l)$ and $F_1\Lambda(K,l)=\Lambda(K,l)$.
 Then the graded algebra $\gr^F\Lambda(K,l)$ is the supertensor
product of $\Lambda(k,l)$ and the exterior algebra with one
generator $\{\pi\}$, so it remains to apply the module part of
the same Corollary from~\cite{PP} (or more precisely, its
straightforward generalization to the infinite-dimensional setting).
\end{proof}

\begin{thm2} \
\begin{enumerate}
\renewcommand{\theenumi}{\arabic{enumi}}
\item
 Let $K$ be an algebraic extension of\/ $\R$, \ $\Q_p$, or\/
$\F_p((z))$, and\/ $l$~be a prime number.
 Let $c$ be an element of\/ $\KM_1(K)/l$.
 Then the ideal $(c)\sub\KM(K)/l$ is a Koszul module over\/
$\KM(K)/l$.
\item
 Let $K$ be a Henselian discrete valuation field with the residue
field\/~$k$ and\/ $l\ne\chr k$ be a prime number.
 Assume that the graded algebra $\KM(k)/l$ is Koszul and for any
element $c\in\KM_1(k)/l$ the ideal $(c)\sub\KM(k)/l$ is a Koszul
module over\/ $\KM(k)/l$.
 Then the graded algebra $\KM(K)/l$ has the same properties.
\end{enumerate}
\end{thm2}

\begin{proof}
 Part~(1): the case of an infinite algebraic extension is deduced
from that of a finite extension by passing to an inductive limit.
 When $K\supset\R$, \ $l=\chr K$, or $K$ does not contain a primitive
$l$\+root of unity, the assertion is trivial.
 The assertion is also trivial when $c=0$.
 So let us assume that $K$ is a finite extension of $\Q_p$ or
$\F_p((z))$ containing a primitive $l$\+root of unity and $c\ne0$.

 When $\{c,c\}=0$ in $\KM_2(K)/l$, choose any ordered basis of
$K^*/K^{*l}$ starting with $x_0=c$ (for simplicity, one can also
pick the next basis vector $x_1$ so that $x_0x_1\ne0$) and use
the result of~\ref{annihilator-koszul}.
 This covers the cases when $l$~is odd or $K$ contains a square root
of~$-1$.
 When $\{c,c\}\ne0$ in $\KM_2(K)/2$ but $c\ne\{-1\}$ in $\KM_1(K)/2$,
choose an ordered basis of $K^*/K^{*2}$ starting with $x_0$, $x_1$
such that $x_0^2=0$, \ $x_1=c$, and $x_0x_1\ne0$ in $\KM_2(K)/2$.
 Then for the corresponding $\Gamma$\+valued filtration $F$ on
$\KM(K)/2$ the ideal $\gr^F(c)\sub\gr^F\.\KM(K)/2$ is generated
by the element $c$, so one can argue as in~\ref{annihilator-koszul}.

 It remains to consider the case when $c=\{-1\}$ in $\KM_1(K)/2$
and $\{-1,-1\}\ne0$ in $\KM_2(K)/2$.
 Choose an ordered basis of $K^*/K^{*2}$ starting from $x_0$,
$x_1$, $x_2$ such that $x_0^2=0=x_1^2$, \ $x_0x_1\ne0$, and
$x_2=\{-1\}$ (or $x_0^2=0$, \ $x_1=\{-1\}$, and $x_0x_2\ne0$).
 Consider the related $\Gamma$\+valued filtration on $\KM(K)/2$.
 Let us define a $\Gamma$\+valued filtration on the ideal $(\{-1\})$
that is compatible with the action of $\KM(K)/2$ on $(\{-1\})$
but is not induced by the embedding $(\{-1\})\sub\KM(K)/2$.

 Namely, choose any $\Gamma$\+valued filtration $F$ on the degree~$1$
component of the ideal $(\{-1\})$ and extend it to the degree~$2$
component in such a way that the $\gr^F\.\KM(K)/2$\+module
$\gr^F(c)$ be generated by its degree~$1$ component
(cf.~\ref{algebras-koszulity}).
 Then the $\gr^F\.\KM(K)/2$\+module $\gr^F(c)$ is isomorphic to
the quotient module of the quadratic commutative monomial algebra
$\gr^F\.\KM(K)/2$ by the ideal generated by all the $x_\alpha$
except $x_2$ (resp., $x_1$), so it remains to use the result
of~\cite[proof of Theorem~8.1 from Chapter~4]{PP}.
 It is essential here that $x_2^2=0$ (resp., $x_1^2=0$)
in $\gr^F\.\KM(K)/2$.

 Part~(2): recall that $\KM(k)/l$ can be naturally considered as
a subalgebra of $\KM(K)/l$.
 If $c\in\KM_1(k)/l$, consider the filtration $F$ on $\KM(K)/l$
defined in the proof of part~(2) of Theorem~1 and the induced
filtration on the ideal~$c\.\KM(K)/l$.
 Then $\gr^F\.\KM(K)/l$ is the supertensor product of $\KM(k)/l$
with the exterior algebra with one generator in degree~$1$, and
the $\gr^F\.\KM(K)/l$\+module $\gr^F c\.\KM(K)/l$ is the supertensor
product of $\gr^F c\.\KM(k)/l$ with the same exterior algebra.
 So it remains to apply~\cite[Corollary~1.2 of Chapter~3]{PP}.

 If $c\in\KM_1(K)/l$ but $c\notin\KM_1(k)/l$, one can assume that
$c=\{\pi\}$ in the notation of~\ref{discrete-valuation}.
 In this case the ideal $(c)\sub\KM(K)/l$ is Koszul whenever
the algebra $\KM(k)/l$ is Koszul.
 It suffices to consider the same filtration $F$ on $\KM(K)/l$
and the induced filtration on the ideal~$(\{\pi\})$.
 The $\gr^F\.\KM(K)/l$\+module $\gr^F(\{\pi\})$ is the supertensor
product of the $\KM(k)/l$\+module $\KM(k)/l$ and the trivial
one-dimensional module over the exterior algebra with one generator.
\end{proof}

\Section{Module Koszulity in Symplectic Case}
\label{symplectic-koszulity}

 Let $l$ be a prime number, and $K$ be an algebraic extension of $\Q$
or $\F_q(z)$ containing a primitive $l$\+root of unity if $l$~is odd,
or containing a square root of~$-1$ if $l=2$.
 Let $\Lambda(K,l)$ denote the exterior algebra generated by
the $\Z/l$\+vector space $K^*/K^{*l}$, and $J_K$ denote the kernel
of the morphism of graded algebras $\Lambda(K,l)\rarrow\KM(K)/l$.

\begin{thm}
 The ideal $J_K$ is a Koszul module over $\Lambda(K,l)$
(in the grading shifted by~$1$).
 In other words, the $\Lambda(K,l)$\+module $\KM_+(K)/l$ is Koszul.
\end{thm}

\begin{proof}
 Passing to the inductive limit of finite extensions of $\Q$ or
$\F_q(z)$ containing the needed root of unity, one reduces
the problem to the case when $K$ is such a finite extension.
 So let $K$ be a finite extension of $\Q$ or $\F_q(z)$ containing
a primitive $l$\+root of unity if $l$ is odd, or containing
a square root of~$-1$ if $l=2$.

 Apply Lemma~\ref{symplectic-case} to the case of the symplectic
vector space $W_S$ decomposed into the orthogonal direct sum of
symplectic subspaces $K_v^*/K_v^{*l}$, \ $v\in S$, and the Lagrangian
subspace $K_S/l\sub W_S$ (see~\ref{exceptional-set}).
 Let $M_v\sub K_v^*/K_v^{*l}$ be the Lagrangian subspaces so obtained.
 The restriction of the form $({-},{-})_S$ defines a nondegenerate
pairing between $K_S/l$ and $\bigoplus_{v\in S}M_v$, which allows to
identify $K_S/l$ with the direct sum of the dual spaces
$\bigoplus_{v\in S}M_v^*$.
 So we have constructed a direct sum decomposition of $K_S/l$ indexed
by $v\in S$; let $\{\.b_i: i=0,\dsc,\#S-1\.\}$ be a basis in $K_S/l$
whose elements belong to the direct summands of this decomposition.
 Let $b_i^*$ denote the dual basis in $\bigoplus M_v$; introduce
the notation $b_i^*\in M_{v(i)}$.
 Notice that the image of $b_i$ in $K_v^*/K_v^{*l}$ belongs to $M_v$
for all $v\ne v(i)$.

 For any divisor $D$ of $K$ outside $S$ there exists a unique element
$a_D\in K^*/K_S^l$ whose divisor outside $S$ is equal to $D$ and whose
image in $W_S$ belongs to $\bigoplus M_v$.
 In particular, for any valuation $p$ of $K$ outside $S$ there exists
a unique element $a_p\in K^*/K_S^l$ with this property, whose divisor
outside $S$ is equal to~$p$.
 The pairing with the image of $a_p$ in $\bigoplus M_v$, as a linear
function $K_S/l\rarrow\mu_l$, coincides with the Frobenius element
of~$p$ in $\Gal(K[\sqrt[l]{K_S}]/K)$
(see Lemma~\ref{exceptional-set}.3).

 Choose a well-ordered basis of $K^*/K^{*l}$ consisting of
the elements $b_i$ and $a_p$ (in any order).
 Consider the related $\Gamma$\+valued filtration $F$ on $\KM(K)/l$
and pass to the associated quotient monomial algebra
$\gr^F\.\KM(K)/l$.
 The graph $T$ of nonzero quadratic monomials in the latter algebra
contains no cycles, so the assertion of Theorem follows from
Proposition~\ref{relations-koszul}(2).
 Indeed, it suffices to notice that one can assign a valuation to
every basis element in this basis so that the product of any two
basis elements can only have nonzero components (see~\ref{with-root})
at the two valuations corresponding to the two basis vectors being
multiplied.
 Thus for any elements $x_1$, $x_2$,~\ldots, $x_n$ in this basis
the products $x_1x_2$, $x_2x_3$,~\dots, $x_{n-1}x_n$, $x_nx_1$
cannot be linearly independent in $\KM_2(K)/l$ (one also has to take
into account the reciprocity law).

 To obtain a more explicit PBW-basis, choose for each $i=1$,~\dots,
$\#S-1$ a valuation $p_i$ of $K$ outside $S$ such that the Frobenius
element of $p_i$ in $\Gal(K[\sqrt[l]{K_S}]/K)$ is equal to
the pairing with $b_0^*+b_i^*$, while its Frobenius element in
$\Gal(K[\sqrt[l]{a_{p_1}}, \dsc,\sqrt[l]{a_{p_{i-1}}}]/K)$ is trivial.
 Denote by~$q$ those valuations of $K$ outside of $S$ and $\{p_i\}$
whose Frobenius elements in $\Gal(K[\sqrt[l]{K_S}]/K)$ are equal to
the pairing with $b_0^*$, while the Frobenius elements in
$\Gal(K[\sqrt[l]{a_{p_1}},\dsc,\sqrt[l]{a_{p_{\#S-1}}}]/K)$ are trivial.
 Denote by~$r$ the remaining valuations.
 Notice that for each valuation~$r$ there exists a valuation~$q$
whose Frobenius element in $\Gal(K[\sqrt[l]{a_r}])$ is nontrivial.

 Choose the following well-ordered basis of $K^*/K^{*l}$
$$
 b_0,\ a_{p_1},\dsc,a_{p_{\#S-1}},\ a_q,\ b_1,\dsc,b_{\#S-1},\ a_r,
$$
where the ordering between $a_q$ and between $a_r$ is arbitrary.
 Then the set of surviving quadratic monomials $T$ will consist of
all the monomials $b_0 a_{p_i}$ and $b_0 a_q$, \ some of the monomials
$b_0b_i$ or $a_{p_i}b_i$ (exactly one monomial of one of these forms
for every nonarchimedean valuation $v(i)\ne v(0)$ in $S$), and some
of the monomials $b_0a_r$, \ $a_{p_i}a_r$, or $a_qa_r$ (exactly one
monomial of one of these forms for every valuation~$r$).
\end{proof}

\begin{rem}
 One may wish to extend the above result to the global fields not
necessarily containing a square root of~$-1$ when $l=2$ in the way
suggested by Theorem~\ref{local-koszulity}.1.
 There is the following obstacle, however.
 If one tries to argue as in the proof of
Theorem~\ref{local-koszulity}.1, one has to find a well-ordered basis
of $K^*/K^{*2}$ that defines a PBW-basis for \emph{both} algebras
$\Lambda(K,2)$ and $\KM(K)/2$.
 But constructing a PBW-basis for $\Lambda(K,2)$ requires putting
the element~$-1$ near the bottom of the $\Gamma_1$\+valued filtration
on $K^*/K^{*2}$, while constructing a PBW-basis for $\KM(K)/2$
requires putting all elements of $K^*/K^{*2}$ that are negative at
some real valuations near the top of that filtration
(cf.~the construction in Section~\ref{general-koszulity}).
\end{rem}

\Section{Algebra Koszulity in General Case}  \label{general-koszulity}

 Let $l$ be a prime number, and $K$ be an algebraic extension of
$\Q$ or $\F_q(z)$ containing a primitive $l$\+root of unity.

\begin{thm}
 The graded algebra $\KM(K)/l$ is Koszul.
\end{thm}

\begin{proof}
 As above, one can assume that $K$ is a finite extension of $\Q$
or $\F_q(z)$.
 The case when $l$~is odd follows from the result of Section~4 and
\cite[Corollary~6.2(c)]{Pbogom}, so we will implicitly assume that
$l=2$ (though this is not necessary).

 Choose a set of exceptional valuations $S$ for the field~$K$
satisfying a slightly stronger condition than
in~\ref{exceptional-set}: namely, let it be additionally required
that the nonarchimedean valuations in $S$ generate the extended
class group of~$K$ (i.~e., the class group defined taking into
account the signs of elements of $K^*$ at the real valuations).
 Let $K_S^+\sub K_S$ and $K^+\sub K^*$ denote the subgroups of all
elements that are positive at all the real valuations.
 Then one has $K^*=K^+K_S$.

 For each nonarchimedean valuation $s\in S$ pick an element
$w_s\in K_s^*/K_s^{*l}$ orthogonal to the class of~$-1$ in
$K_s^*/K_s^{*l}$ with respect to the pairing $\{{-},{-}\}_s$,
and consider $w_s$ as an element of $W_S$.
 We need the pairings with the elements $w_s$ to define nonzero
linear functions $K_S^+/K_S^l\rarrow\mu_l$.
 Enlarging, if it be necessary, the set $S$, one can always choose
such elements~$w_s$.

 Indeed, one only has to use the weak approximation
theorem~\cite[Section~II.6]{CF} in order to find a finite set
of valuations $S'\supset S$ such that the pairings with the given
elements $w_s$ define nonzero linear functions $K_{S'}^+/K_{S'}^l
\rarrow\mu_l$ for all $s\in S$.
 Now for any $u\in S'\setminus S$ there is an element $d\in K_{S'}^+$
with the logarithmic valuation $u(d)$ not divisible by~$l$, because
nonarchimedean valuations in $S$ generate the extended class group.
 Since $u$~is a nonarchimedean valuation not lying over~$l$,
choosing $w_u$ to be the class of an element $b\in K_u^*$ with
with $u(b)=0$ and $b\notin K_u^{*l}$ guarantees $\{-1,w_u\}_u=0$ and
$\{w_u,d\}_u\ne0$, as desired.

 Let $p$~be a valuation of $K$ outside $S$ such that the Frobenius
element of~$p$ in $\Gal(K[\sqrt[l]{K_S}]/K)$ is trivial.
 Then by Lemma~\ref{exceptional-set}.3 there exists an element
$a_p\in K^*$ whose divisor outside $S$ is equal to~$p$ and
whose image in $W_S$ is zero.
 For each nonarchimedean valuation $u\in S$, pick a valuation~$q_u$
outside $S$ whose Frobenius in $\Gal(K[\sqrt[l]{K_S}]/K)$ as a linear
function $K_S\rarrow\mu_l$ is equal to the pairing with~$w_u$, while
the Frobenius in $\Gal(K[\sqrt[l]{a_p}]/K)$ is nontrivial.
 All the valuations~$q_u$ must be different.
 By the same lemma, there exists an element $a_{q_u}\in K^*$ whose
divisor outside $S$ is equal to~$q_u$ and whose image in $W_S$
is equal to~$w_u$.
 By the definition, the elements $a_p$ and $a_{q_u}$ belong to
$K^+$, and one has $\{a_p,a_p\}=0$ in $\KM_2(K)/l$.

 For each valuation~$r$ of the field $K$ outside of $S$, \ $p$,
and~$q_u$, choose an element $a_r\in K^+$ whose divisor outside
of~$S$ is equal to~$r$.
 Let us denote by~$r'$ those valuations~$r$ whose Frobenius element
in $\Gal(K[\sqrt[l]{a_p}]/K)$ is nontrivial and by~$r''$ the remaining
ones.
 For each real valuation~$v$ pick an element $a_v\in K_S$ that is
negative at~$v$ and positive at all the other real valuations.
 Choose any basis $k_j$ in $K_S^+/K_S^l$.

 Consider the following well-ordered basis of $K^*/K^{*l}$
$$
 a_p,\ a_{q_u},\ k_j,\ a_{r'},\ a_{r''},\ a_v,
$$
where the order within each group can be arbitrary.
 Consider the related $\Gamma$\+valued filtration $F$ on $\KM(K)/l$
and the associated quotient algebra $\gr^F\.\KM(K)/l$.
 The set of surviving quadratic monomials $T$ consists of
all the monomials $a_p a_{q_u}$ and $a_p a_{r'}$, some monomials
of the form $a_{q_u} k_j$ (exactly one such monomial for every~$u$),
some monomials of the forms $a_{q_u}a_{r''}$, \ $k_ja_{r''}$,
or $a_{r'}a_{r''}$ (exactly one monomial of one of these forms for
every~$r''$), and all the monomials $a_v^2$.

\medskip
 To prove these assertions, introduce the notion of
the \emph{support} of an element $\alpha\in\KM_2(K)/l$, defined
as the set of all valuations $y$ such that the image of~$\alpha$
in $\KM_2(K_y)/l$ is nontrivial.
 The subspace of $\KM_2(K)/l$ consisting of all the elements supported
inside a set of valuations $Y$ has the dimension equal to the number of
noncomplex valuations in $Y$ minus one (see~\ref{with-root}).

 Let us discuss all the quadratic monomials in our basis in the order
of their increase.
 The product $\{a_p,a_{q_u}\}$ is nonzero in $\KM_2(K)/l$ and
supported in $p$ and~$q_u$.
 Likewise, the product $\{a_p,a_{r'}\}$ is nontrivial and supported
in $p$ and~$r'$.
 Taken together, these products generate the whole subspace of all
elements supported inside the set of valuations~$p$, \ $q_u$, and~$r'$.
 Every element divisible by~$a_p$ in $\KM_2(K)/l$ is supported
inside this set, hence the products $\{a_p,a_v\}$ are linear
combinations of smaller monomials with respect to our ordering.
 The products $\{a_p,k_j\}$ and $\{a_p,a_{r''}\}$ vanish
in $\KM_2(K)/l$.

 A product of the form $\{a_{q_{u_1}},a_{q_{u_2}}\}$, where $u_1$ and
$u_2$ belong to the set of valuations~$u$, is supported inside
the set of two valuations $q_{u_1}$ and $q_{u_2}$, so it is a linear
combination of smaller monomials.
 Indeed, this holds for $u_1\ne u_2$, since $\{w_{u_1},w_{u_2}\}_s=0$
for all $s\in S$, and one actually has $\{a_{q_u},a_{q_u}\}=0$ in
$\KM_2(K)/l$ for $u_1=u=u_2$, because $\{w_u,w_u\}_u=\{-1,w_u\}_u=0$.

 A product of the form $\{a_{q_u},k_j\}$ either vanishes or is supported
in $u$ and~$q_u$, and there exists at least one nonvanishing product
of such form for every~$u$.
 Taken together with the monomials containing $a_p$, this product
generates the subspace of all elements supported inside the set of
valuations in the above list together with the valuation~$u$.
 The products $\{a_{q_u},a_{r'}\}$ and $\{a_{q_u},a_v\}$ are contained
in this subspace, so they are linear combinations of smaller monomials.
 Taken together for all~$u$, the products we have mentioned up to
this point generate the subspace of all elements supported inside
the set of valuations~$p$, \ $q_u$, \ $r'$, and all the nonarchimedean
valuations from~$S$.
 The products $\{k_{j_1},k_{j_2}\}$, \ $\{k_j,a_{r'}\}$, \
$\{k_j,a_v\}$, \ $\{a_{r'_1},a_{r'_2}\}$, and $\{a_{r'},a_v\}$ are
contained in this subspace, so they are also linear combinations of
smaller monomials.

 The product $\{a_{q_u},a_{r''}\}$ is supported inside the set of
three valuations~$u$, \ $q_u$, and $r''$, so one can easily see
that at most one such product belongs to the set of surviving
monomials~$T$, and this can only happen if the support of this
product contains~$r''$.
 On the other hand, recall that it is only the $\Gamma_1$\+valued
filtration on $\KM_1(K)/l$ rather than the basis itself that
determines the set~$T$. 
 The filtration does not change if we assume that for those
valuations~$r'$ whose Frobenius element is trivial in
$\Gal(K[\sqrt[l]{K_S}]/K)$ the element $a_{r'}$ is chosen in such
a way that its image in $W_S$ is trivial.

 By Chebotarev's density theorem applied to the field extension
$K[\sqrt[l]{K_S,a_p,a_{r''}}]/K$, for every valuation~$r''$ there
exists a valuation~$r'$ with the above property such that
the Frobenius element of~$r'$ in $\Gal(K[\sqrt[l]{a_{r''}}]/K)$
is nontrivial.
 Then the product $\{a_{r'},a_{r''}\}$ in $\KM_2(K)/l$ is supported
in $r'$ and~$r''$, and nonzero.
 For every~$r''$, the set $T$ contains the minimal of the products
$\{a_{q_u},a_{r''}\}$, \ $\{k_j,a_{r''}\}$, and $\{a_{r'},a_{r''}\}$
whose support contains~$r''$; the above argument shows that such
a monomial exists.

 Taken together, the products listed up to this point generated
the whole subspace of all elements in $\KM_2(K)/l$ supported
outside of the real valuations~$v$.
 The products $\{a_{r''_1},a_{r''_2}\}$, \ $\{a_{r''},a_v\}$, and
$\{a_{v_1},a_{v_2}\}$ for $v_1\ne v_2$ belong to this subspace,
so they are linear combinations of smaller monomials.
 The support of the product $\{a_v,a_v\}$ contains~$v$ and does not
contain any other real valuations, so all the monomials of this type
belong to~$T$.

 The set/graph $T$ contains no triangles and no monomials divisible
by $a_v$ except $a_v^2$.
 The elements $a_v^n$ are obviously linearly independent in
$\KM_n(K)/l$ for all $n\ge1$, so one readily checks that
the algebra $\gr^F\.\KM(K)/l$ is quadratic.
 Consequently the algebra $\KM(K)/l$ is Koszul
(see \ref{filtered-algebras-modules}
and~\ref{algebras-koszulity}).

\medskip
 Alternatively, one can write after the elements $a_{q_u}$ in
the above well-ordering the elements $a_{q'}$ with zero images
in $W_S$ corresponding to the valuations $q'$ outside of $S$ and~$p$
whose Frobenius elements in $\Gal(K[\sqrt[l]{K_S}]/K)$ are trivial
and in $\Gal(K[\sqrt[l]{a_p}]/K)$ are nontrivial.
 In this approach, one does not introduce the distinction between
$r'$ and~$r''$, but instead excludes the valuations~$q'$ from
the list of valuations~$r$.
 Then the set of surviving quadratic monomials $T$ will consist of
all the monomials $a_p a_{q_u}$ and $a_p a_{q'}$, some monomials
of the form $a_{q_u} k_j$ (exactly one such monomial for every~$u$),
some monomials of the forms $a_p a_r$, \ $a_{q_u} a_r$, or
$a_{q'}a_r$ (exactly one monomial of one of these forms for
every~$r$), and all the monomials $a_v^2$.
\end{proof}

\Section{Koszulity of Annihilator Ideals}

 Let $l$ be a prime number, $K$ be an algebraic extension of $\Q$
or $\F_q(z)$ containing a primitive $l$\+root of unity, and
$c\in K^*/K^{*l}$ be an element such that $\{c,c\}=0$ in
$\KM_2(K)/l$.
 In particular, when $l$ is odd, or $l=2$ and $K$ contains a square
root of~$-1$, the element~$c$ can be arbitrary.

\begin{thm}
 The ideal $(c)=c\.\KM(K)/l\sub\KM(K)/l$ is a Koszul module over
the Koszul algebra $\KM(K)/l$.
\end{thm}

\begin{proof}
 The argument below is a variation of the proof in
Section~\ref{general-koszulity}.
 As above, we can assume that $K$ is finite over $\Q$ or $\F_q(z)$;
we can also assume that $c\ne 0$ in $K^*/K^{*l}$.

 Choose a set of exceptional valuations $S$ for the field $K$
satisfying the conditions of Section~\ref{general-koszulity} and
containing the divisor of the element~$c$.
 Choose an element $a_p\in K^*$ whose divisor outside $S$ is equal to
a certain valuation~$p$ and whose image in $W_S$ is zero.
 For each nonarchimedean valuation $u\in S$ such that the image of~$c$
in $K_u^*/K_u^{*l}$ is zero, choose an element $w_u\in K_u^*/K_u^{*l}$
such that $\{w_u,-1\}_u=0$ and the pairing with $w_u$ is
a nonzero linear function $K_S^+/K_S^l\rarrow\mu_l$.
 Pick an element $a_{q_u}\in K^*$ whose divisor outside $S$ is equal
to a certain valuation~$q_u$ and whose image in $W_S$ is equal
to~$w_u$.
 We also need the Frobenius element of $q_u$ to be nontrivial in
$\Gal(K[\sqrt[l]{a_p}]/K)$ and all the valuations $q_u$ to
be different.
 Clearly, one has $a_p$ and $a_{q_u}\in K^+$ and $\{c,a_p\}=
\{c,a_{q_u}\}=\{a_p,a_p\}=0$ in $\KM(K)/l$.

 For each valuation~$r$ outside of $S$, \ $p$ and~$q_u$, choose
an element $a_r\in K^+$ whose divisor outside of $S$ is equal to~$r$.
 Denote by~$r'$ those valuations~$r$ whose Frobenius element in
$\Gal(K[\sqrt[l]{a_p}]/K)$ is nontrivial and by~$r''$ the remaining
ones.
 For each real valuation~$v$ pick an element $a_v\in K_S$ that is
negative at~$v$ and positive at all the other real valuations.
 Notice that $c\in K_S^+/K_S^l$; let elements $k_j\in K_S^+$
complement the element~$c$ to a basis of $K_S^+/K_S^l$.

 Consider the following well-ordered basis of $K^*/K^{*l}$
$$
 c,\ a_p,\ a_{q_u},\ k_j,\ a_{r'},\ a_{r''},\ a_v,
$$
where the ordering within each group can be arbitrary.
 The related set $T$ of surviving quadratic monomials consists of
all the monomials $a_p a_{q_u}$, some monomials of the forms $c k_j$
and $c a_{r'}$, some monomials of the form $a_{q_u}k_j$ (exactly one
such monomial for every valuation~$u$), some monomials $a_p a_{r'}$,
some monomials of the forms $c a_{r''}$, \ $a_{q_u} a_{r''}$, \
$k_j a_{r''}$, or $a_{r'}a_{r''}$ (exactly one monomial of one
of these forms for every~$r''$), and all the monomials $a_v^2$.

\medskip
 Let us prove these assertions.
 Denote by $Y$ the set of all valuations $y$ of $K$ for which
$c\in K_y^*\setminus K_y^{*l}$.
 First we will have to show that the products $\{c,k_j\}$ and
$\{c,a_{r'}\}$ generate the subgroup of all elements in $\KM_2(K)/l$
supported inside the set of those valuations $y\in S$ or $y=r'$ that
belong to~$Y$.
 Specifically, for every element $w\in W_S$ such that $(c,w)_S=0$
let us consider a valuation $r'$ such that the image of $a_{r'}$
in $W_S$ belongs to $w+K_S/l$.
 Then $\{c,a_{r'}\}_{r'}=0$, hence $r'\notin Y$, and the support of
$\{c,a_{r'}\}$ is contained in~$S$.

 The products $\{c,a_{r'}\}$ for such valuations~$r'$ generate
the subgroup of all elements in $\KM_2(K)/l$ supported inside the set
of valuations $S\cap Y$.
 Indeed, the subspace of vectors of the form $(\{c_s,w_s\}_s)_{s\in S}$
in $\bigoplus_{s\in S}\mu_l$, where $w=(w_s)_{s\in S}$ runs over all
the elements in $W_S$ for which $\sum_{s\in S}\{c_s,w_s\}_s=0$,
consists precisely of those vectors that belong to the kernel of
the summation map $\bigoplus_{s\in S}\mu_l\rarrow\mu_l$
and are supported inside the set of all places $s\in S$ at which
the component $c_s\in K_s^*/K_v^{*l}$ is nonzero.
 On the other hand, for a valuation $r'$ belonging to $Y$ the support
of the product $\{c,a_{r'}\}$  is contained in $S\cup\{r'\}$ and
contains~$r'$.

 The product $\{c,a_{r''}\}$ is supported inside $S$ and~$r''$, so
it belongs to the set of surviving monomials $T$ if and only if
its support contains~$r''$, that is $r''\in Y$.
 The products $\{c,k_j\}$, \ $\{c,a_{r'}\}$, and $\{c,a_{r''}\}$
generate the subgroup of all elements supported inside the set $Y$,
so the products $\{c,a_v\}$ are linear combinations of smaller
monomials in the ordering.
 The products $\{a_p,k_j\}$ and $\{a_p,a_{r''}\}$ vanish in
$\KM_2(K)/l$.
 The product $\{a_p, a_{q_u}\}$ is nonzero and supported in the two
valuations $p$ and~$q_u$, which do not belong to~$Y$; hence this
monomial belongs to~$T$.
 Likewise, the product $\{a_p,a_{r'}\}$ is nontrivial and supported
in $p$ and~$r'$, hence it belongs to $T$ whenever $r'\notin Y$;
of all the products $\{a_p,a_{r'}\}$ with $r'\in Y$, it is only
the smallest one that belongs to~$T$.

 Taken together, the products mentioned up to this point generate
the subgroup of $\KM_2(K)/l$ supported inside the set of
valuations $p$, \ $q_u$, \ $r'$, and all valuations from~$Y$.
 The rest of the argument is very similar to the one in
Section~\ref{general-koszulity}, the only difference being that
all the valuations from $Y$ have been already ``covered''.

\medskip
 The set/graph $T$ contains no triangles and no monomials divisible
by $a_v$ except $a_v^2$, so one readily checks that the algebra
$\gr^F\.\KM(K)/l$ is quadratic.
 By the result of~\ref{annihilator-koszul}, the ideal
$(c)\sub\KM(K)/l$ is a Koszul module over $\KM(K)/l$.
\end{proof}

\Section{Fields without the Root of Unity}

 Let $K$ be an algebraic extension of $\Q$ or $\F_q(z)$ and $l$~be
a prime number such that either $l=\chr K$ or $K$ contains no
primitive $l$\+root of unity.
 Let $\Lambda(K,l)$ be the exterior algebra generated by
$\Lambda_1(K,l)=K^*/K^{*l}$ and $J_K\sub\Lambda(K,l)$ be the kernel
of the map of graded algebras $\Lambda(K,l)\rarrow\KM(K)/l$.
 Let $c\in\KM_1(K)/l$ be an element and $(c)\sub\KM(K)/l$ be
the ideal generated by~$c$.

\begin{thm}
 The ideal $J_K$ is a Koszul $\Lambda(K,l)$\+module (in
the grading shifted by~$1$).
 The ideal~$(c)$ is a Koszul module over a Koszul algebra
$\KM(K)/l$.
\end{thm}

\begin{proof}
 The case $l=\chr K$ is trivial (see~\ref{equal-char} and
the beginning of the proof of Theorem~\ref{local-koszulity}.1).
 It also suffices to consider the case when $K$ is finite over $\Q$
or $\F_q(z)$.
 
 Let $S'$ be an exceptional set of valuations of~$K$ satisfying
the conditions of~\ref{without-root}.
 To prove the first assertion of Theorem, for each nonarchimedean
valuation $u\in S'$ such that $K_u$ contains a primitive $l$\+root
of unity choose a pair of elements $w'_u$, $w''_u\in K_u^*/K_u^{*l}$
such that $\{w'_u,w''_u\}_u\ne0$ in $\KM_2(K_u)/l$.
 Using Lemma~\ref{without-root}, choose valuations $p'_u$ and $p''_u$
of $K$ such that all of them are different, do not belong to $S'$,
the completions $K_{p'_u}$ and $K_{p''_u}$ do not contain a primitive
$l$\+root of unity, and there exist elements $a_{p'_u}$ and
$a_{p''_u}\in K^*$ whose images in $W'_{S'}$ are equal to $w'_u$ and
$w''_u$, and whose divisors outside $S'$ are equal to $p'_u$
and~$p''_u$.

 Choose a numbering by nonnegative integers for all the valuations~$r$
of $K$ outside $S'$ such that $K_r$ contains a primitive $l$\+root
of unity, and order them according to this numbering.
 By induction on this order, choose for each such valuation~$r$
a valuation $q=q(r)$ outside of $S'$ and all $p'_u$, $p''_u$ such that 
$K_q$ does not contain a primitive $l$\+root of unity, the valuations
$q(r)$ are different for different~$r$, and there exists an element
$a_q\in K^*$ whose divisor outside $S'$ is equal to~$q$, whose image
in $W'_{S'}$ and in $K_{r'}/K_{r'}^{*l}$ is zero for all $r'<r$,
and whose image in $K_r/K_r^{*l}$ is nontrivial.
 The existence of such a valuation~$q(r)$ follows from
Lemma~\ref{without-root} applied to the set
$S'\cup\{\.r'\mid r' < r\.\}\cup\{r\}$ in place of~$S'$.

 Finally, for all the valuations $v$ of $K$ outside of the sets $S'$, \
$\{p'_u\}$, \ $\{p''_u\}$, and~$\{q(r)\}$ choose elements $a_v\in K^*$
whose divisors outside $S'$ are equal to~$v$.
 Let $k_j$ be any basis of $K_{S'}/K_{S'}^l$.
 Consider the following well-ordered basis of $K^*/K^{*l}$
$$
 a_{q(r)},\ a_{p'_u},\ a_{p''_s},\ k_j,\ a_v,
$$
where the ordering of $a_{q(r)}$ is according to the ordering of~$r$,
while the ordering within each of the other groups is arbitrary.
 The related set $T$ of surviving quadratic monomials consists of
the monomials $a_{q(r)}a_r$ and $a_{p'_u}a_{p''_u}$.
 Not only this graph does not contain any cycles, but there is even
no vertex adjacent to more than one edge.
 By Proposition~\ref{relations-koszul}(2), the ideal $J_K$ is Koszul.
 It follows that the graded algebra $\KM(K)/l$ is Koszul, too.

 Now let us prove the second assertion.
 We can assume that the element~$c$ is nonzero in $K^*/K^{*l}$.
 Let $S'$ be an exceptional set of valuations satisfying
the conditions of~\ref{without-root} and containing the divisor of~$c$. 
 For each nonarchimedean valuation $u'\in S'$ such that $K_{u'}$
contains a primitive $l$\+root of unity and the image of~$c$ in
$K_{u'}^*/K_{u'}^{*l}$ is nonzero, choose an element $w_{u'}\in
K_{u'}^*/K_{u'}^{*l}$ such that $\{w_{u'},c\}_{u'}\ne0$.
 For each of the remaining nonarchimedean valuations $u''\in S'$
such that $K_{u''}$ contains a primitive $l$\+root of unity,
choose a pair of elements $w'_{u''}$, $w''_{u''}\in K_{u''}$ such
that $\{w'_{u''},w''_{u''}\}_{u''}\ne0$.
 
 Choose valuations $p_{u'}$, \ $p'_{u''}$, and $p''_{u''}$ outside
of $S'$ such that all of them are different, the corresponding
completions do not contain a primitive $l$\+root of unity, and
there exist elements $a_{p_{u'}}$, \ $a_{p'_{u''}}$, and $a_{p''_{u''}}
\in K^*$ whose divisors outside $S'$ are equal to these valuations
and whose images in $W'_{S'}$ are equal to $w_{u'}$, \ $w'_{u''}$,
and $w''_{u''}$.
 For each valuation $r$ outside $S'$ such that $K_r$ contains
a primitive $l$\+root of unity, choose a valuation $q(r)$ outside
of $S'$, \ $p_{u'}$, \ $p'_{u''}$, and $p''_{u''}$ such that 
the valuations $q(r)$ are different for different~$r$, the completion
$K_{q(r)}$ does not contain a primitive $l$\+root of unity, and there
exists an element $a_{q(r)}\in K^*$ whose divisor outside $S'$ is equal
to $q(r)$ and whose image in $K_r^*/K_r^{*l}$ is nonzero.

 Finally, for all the valuations $v$ of $K$ outside of the sets $S'$, \
$\{p_{u'}\}$, \ $\{p'_{u''}\}$, \ $\{p''_{u''}\}$, and~$\{q(r)\}$
choose elements $a_v\in K^*$ whose divisors outside $S'$ are equal
to~$v$.
 Let elements $k_j\in K_{S'}/K_{S'}^l$ complement~$c$ to a basis of
$K_{S'}/K_{S'}^l$.
 Consider the following well-ordered basis of $K^*/K^{*l}$
$$
 c,\ a_{p_{u'}},\ a_{p'_{u''}},\ a_{p''_{u''}},\ a_{q(r)},\ k_j,\ a_v,
$$
where the ordering within each group can be arbitrary.
 Then the related set $T$ of surviving quadratic monomials consists
of all the monomials $c a_{p_{u'}}$ and $a_{p'_{u''}}a_{p''_{u''}}$
and some monomials of the forms $c a_v$, \ $a_{p_{u'}}a_v$, \
$a_{p'_{u''}}a_v$, \ $a_{p''_{u''}}a_v$, or $a_{q(r)}a_v$
(exactly one monomial of one of these forms for each valuation $v$
outside $S'$ such that $K_v$ contains a primitive $l$\+root of unity,
and no such monomial for all other valuations~$v$).
 The graph $T$ contains no triangles (and not even any cycles),
so the desired assertion follows from
Proposition~\ref{relations-koszul}(1) and the result
of~\ref{annihilator-koszul}.
\end{proof}

\end{document}